\documentclass[a4paper,11pt,leqno]{article}
\usepackage{amssymb, amsmath, amsthm, graphicx, color, epsfig, amsfonts}

\newtheorem{thm}{Theorem}[section]
\newtheorem{lem}[thm]{Lemma}
\newtheorem{cor}[thm]{Corollary}
\newtheorem{prop}[thm]{Proposition}

\newtheorem{rem}{Remark}[section]

\numberwithin{equation}{section}

\renewcommand{\a}{\alpha}
\renewcommand{\b}{\beta}

\newcommand{\de}{\delta}
\newcommand{\fa}{\varphi}
\newcommand{\ga}{\gamma}
\renewcommand{\k}{\kappa}
\newcommand{\la}{\lambda}
\renewcommand{\th}{\theta}
\newcommand{\si}{\sigma}

\renewcommand{\t}{\tau}

\newcommand{\De}{\Delta}
\newcommand{\Ga}{\Gamma}
\newcommand{\La}{\Lambda}

\newcommand{\lan}{\langle}
\newcommand{\ran}{\rangle}

\def\R{{\mathbb{R}}}
\def\N{{\mathbb{N}}}
\def\Z{{\mathbb{Z}}}
\def\T{{\mathbb{T}}}

\allowdisplaybreaks

\title{$L^p$-Boltzmann-Gibbs principle
\\ via  Littlewood-Paley-Stein inequality}

\author{Tadahisa Funaki}

\date{\today } 

\begin{document}
\maketitle

{\it
This article is dedicated to Claudio Landim on the occasion of his 60th birthday.
}

\begin{abstract} 
In this paper, we establish the Boltzmann-Gibbs principle in the $L^p$ sense
by applying the Littlewood-Paley-Stein inequality.  Our model is
an asymmetric Ginzburg-Landau interface model
on a one-dimensional periodic lattice. Assuming convexity of the potential,
we derive detailed error estimates, particularly their dependence on the size 
of the system and the size of the region on which the
sample average is taken.  Notably, the estimates are uniform in the strength of
the asymmetry.
\footnote{
\hskip -6mm 
Beijing Institute of Mathematical Sciences and Applications, 
No.\ 544 Hefangkou, Huairou District, Beijing 101408, China.}
\end{abstract}

\section{Introduction}

The Boltzmann-Gibbs principle, first formulated by Rost \cite{R}
in 1983, was developed to study equilibrium fluctuations in interacting diffusions
in a setting of the hydrodynamic limit; see also \cite{BR}, \cite{DPSW} and \cite{KL}.
This principle is now widely applied in the study of the large-scale behavior
of interacting systems such as the hydrodynamic limit (at the level of the
law of large numbers) in a non-equilibrium situation, the linear fluctuation limit 
(at the level of CLT) around it and also the KPZ nonlinear fluctuation
limit sometimes called the weak universality. See below for a more detailed explanation.

In a large-size randomly evolving interacting system related to statistical physics, 
the large-scale space-time sample average of an observable can be
replaced by its ensemble average due to the ergodic
property of the system.  When the system has a conserved quantity,
its ensembles or the equilibrium states are not unique, but are given as
a set of probability measures, say $\{\mu_u\}_u$ on a configuration space
$\mathcal{X}$ parametrized by a conserved quantity $u$.
We denote an element of $\mathcal{X}$ by $\eta$ and consider a
function $f(\eta)$ on $\mathcal{X}$, which is called an observable.
The spatial shift operator $\t_x$ acts on $\mathcal{X}$.  To clarify our setup,
let $\mathcal{X}=\R^{\Z^d}$ and $\t_x\eta := (\eta(y+x))_{y\in \Z^d}$ for
$x\in \Z^d$ and $\eta=(\eta(y))_{y\in \Z^d}\in \mathcal{X}$.  We will take $d=1$ later.
One would then expect the following replacement of $f$ under the
large-scale space-time average:
$$
f(\t_x\eta) \approx \lan f\ran (\eta^{(\ell)}(x)),
$$
for large $\ell$,
where $\lan f\ran(u) = E^{\mu_u}[f]$ denotes the ensemble average
of $f$ under $\mu_u$, and $\eta^{(\ell)}(x)$ is the sample average of
the variable $\eta=(\eta(y))_{y\in \Z^d}$ on a large box $\La_\ell = [-\ell,\ell]^d\cap
\Z^d$ centered at $x$:
\begin{align} \label{eq:1.1}
\eta^{(\ell)}(x) := \frac1{|\La_\ell|} \sum_{y\in \La_\ell} \eta(x+y).
\end{align}
Furthermore, when considering the system under $\mu_u$, 
we would have the Taylor expansion:
$$
\lan f\ran (\eta^{(\ell)}(x)) \approx \lan f\ran (u) + \lan f\ran' (u) 
(\eta^{(\ell)}(x) -u) + \tfrac12 \lan f\ran'' (u) 
(\eta^{(\ell)}(x) -u)^2 + \cdots,
$$
where $\lan f\ran' $ and $\lan f\ran''$ denote the first and second 
derivatives of $\lan f\ran$ in $u$, assuming $u\in \R$.
Thus, we expect to have a replacement of $f$ with linear and quadratic functions
of $\eta^{(\ell)}$:
\begin{align} \label{eq:1.BG}
f(\t_x\eta)  \approx \lan f\ran (u) + \lan f\ran' (u) 
(\eta^{(\ell)}(x) -u) + \tfrac12 \lan f\ran'' (u) 
(\eta^{(\ell)}(x) -u)^2 + \cdots,
\end{align}
under $\mu_u$ and under the large-scale space-time average.
This replacement, especially its error estimate, is called the 
Boltzmann-Gibbs principle of first-order when we consider up to the linear term, 
or of second-order when we consider up to the quadratic term.  Our main result is
stated in Theorem \ref{BG-2-B} in an equilibrium situation.  It precisely formulates
this replacement, in particular, in the $L^p$ sense.

The Boltzmann-Gibbs principle
of first-order in non-equilibrium situation is useful for studying the hydrodynamic limit.
For non-gradient models, this principle is known as
a gradient replacement, whereby the diverging non-gradient term
is replaced by a linear function of $\eta$ (see \cite{V93}, \cite{FUY},
 \cite{F24}, \cite{FGW24}), and 
it has been efficiently used also for the Glauber-Kawasaki dynamics 
\cite{FvMST} and the Glauber-zero-range process \cite{EFHPS} of gradient type.
The Boltzmann-Gibbs principle of first-order in equilibrium was shown
for non-gradient model at the CLT scaling \cite{Chang}, \cite{F96},  \cite{Lu}.
Non-equilibrium fluctuations were established in \cite{CY}, \cite{FGJS}, 
\cite{jm1}, \cite{jm2}, and the nearly equilibrium case in \cite{FLS}, \cite{gjmm}.

The Boltzmann-Gibbs principle of second-order in equilibrium was shown by
Gon\c{c}alves, Jara and others \cite{GoJ}, \cite{GJS}, \cite{GO} and 
also \cite{DGP}, \cite{BFS}.  They established the weak universality in the KPZ
scaling limit in a weakly asymmetric regime for interacting particle systems
based on the martingale problem approach called the energy solution.
See Yang \cite{Y1}, \cite{Y2}, \cite{Y3} for non-equilibrium case. 
A similar formulation appears in a tagged particle problem and the CLT for additive 
functionals; see \cite{KV}, \cite{KLO}.

The Boltzmann-Gibbs principle has been proved in almost all cases in the $L^1$ or
$L^2$ sense, or in probability sense, except \cite{GJ} showing in the $L^p$ sense in a 
special Gaussian-based setting.  Indeed,
\cite{GJ} shows an $L^p$ estimate that relies on a hypercontractivity estimate
for Gaussian measures, in which the constant depends on the order of the polynomial.
This method does not work in the present setting.  
The difficulty in showing an $L^p$ estimate lies in the treatment of the
$L^{p/2}$ integral of carr\'e du champ that appears in the proof.  
If $p=2$, it can be rewritten as a Dirichlet form allowing for easy  
integration by parts.  For $p>2$, this method does not work.  In the 
present article, we instead apply a Sobolev-type inequality followed by the
Littlewood-Paley-Stein inequality.  It is important that the constant in this inequality
is uniform with respect to the size of
the system being studied.  Littlewood-Paley theory is useful for
extending $L^2$ results to $L^p$, so it is natural to apply it in our situation.
The Littlewood-Paley-Stein inequality refers to the upper and lower bounds 
for the $L^p$-norms of the so-called Littlewood-Paley $G$-functions.
However, we simply
call the associated Sobolev-type inequality the Littlewood-Paley-Stein inequality.

Section \ref{sec:2} introduces our model, that is, the one-dimensional 
asymmetric Ginzburg-Landau interface model. This model was studied in \cite{DGP}
on $\Z$ instead of our $\T_N$.   Under the diffusive scaling,
our main result is stated in Theorem \ref{BG-2-B}.
This theorem provides the Boltzmann-Gibbs principle of the first- and 
second-orders in the $L^p$ sense, that is, \eqref{eq:1.BG} under 
the space-time average
and in an equilibrium situation with $L^p$ error estimates.
Notably, the estimate is uniform in the strength $\ga$ of the asymmetry,
which allows us to apply it to weak asymmetry $\ga N^{-1/2}$ for 
KPZ weak universality, but also includes strong asymmetry $\ga$.
After several preparatory steps in Section \ref{sec:4}, we prove Theorem 
\ref{BG-2-B} in Section \ref{sec:5}.  Appendix \ref{sec:B} is 
devoted to a discussion of the Littlewood-Paley-Stein inequality.

Currently, our discussion is limited to the one-dimensional Ginzburg-Landau 
interface model.  For this model, the Littlewood-Paley-Stein inequality, which plays 
a fundamental role, is available due to the results of Bakry \cite{Ba87}
at least when the potential of the model is convex; see Section \ref{sec:4.3}.
It would be interesting to extend the Littlewood-Paley-Stein inequality to 
a non-convex potential (see Appendix \ref{sec:B}) or to a certain class of
interacting particle systems such as Kawasaki
dynamics.  Related to the latter, \cite{CCH} discusses it on a graph, but
the uniformity of the estimate with respect to the size of the graphs is not studied.

The weak universality of SPDE models was shown by \cite{HQ}, \cite{HS17}, 
\cite{HX},  \cite{KZ25}. These are essentially established as pathwise limits based 
on the theory of regularity structures, though some of them are formulated as
the convergence in law.  The original motivation of the present article
comes from showing the KPZ scaling limit for interacting systems
by means of a pathwise approach, such as the regularity structures or
the paracontrolled calculus, rather than the method of the martingale problem. 
This approach is natural in view of the results and methods for the SPDE models.
We expected the $L^p$-Boltzmann-Gibbs principle would be useful, but
it is unclear at present.  Recently, \cite{HMW} proved the KPZ scaling limit 
for a weakly asymmetric simple exclusion process within
the framework of regularity structures.

\section{One-dimensional Ginzburg-Landau interface model}
\label{sec:2}

Let us introduce our model, the $\eta$-process.  It is essentially
the same as the process of the height differences $u=(u_t(i))_{i\in \Z}$
studied in \cite{DGP}, except we consider it on $\T_N$ instead of
$\Z$.  In a symmetric case, the $\eta$-process
 is the same as that considered in \cite{GPV}, but we introduce
an asymmetry while maintaining stationary measures.  
The $\phi$-process of heights
is a discrete integral in space of the $\eta$-process in the sense of 
\eqref{eq:eta-phi} or in Remark \ref{rem:4.1-A}.  It helps to quickly
calculate the generator and
carr\'e du champ of the $\eta$-process; see Remark \ref{rem:1.3}.
The $\phi$-process in symmetric case was studied in \cite{FS97}, \cite{F05}
in higher dimensions.

\subsection{Height process $\phi_t$}
\label{sec:2.1}

Let $\T_N =\{1,2,\ldots,N\}$ be a one-dimensional discrete lattice
of size $N\in \N$.  We will introduce the modified periodic and
periodic boundary conditions \eqref{eq:2.mpb}, \eqref{eq:eta-p} for the
$\phi$- and $\eta$-variables.  For $\phi= (\phi(x))_{x\in \T_N} \in 
\mathcal{Y}_N:= \R^N\equiv \R^{\T_N}$
called a height function, let a Hamiltonian
\begin{align}  \label{eq:H}
H(\phi)\equiv H_N(\phi) = \sum_{x\in \T_N} V(\phi(x+1)-\phi(x))
\end{align}
be given, where $V=V(\zeta)$ is a $C^2$ function on $\R$ and grows
at least quadratically as $|\zeta|\to\infty$; see \eqref{eq:2.V} below.
In the sum \eqref{eq:H}, $\phi(N+1)$ is defined by the modified periodic boundary
condition $\phi(N+1)= \phi(1)+ Nm$, which can be naturally extended to $\Z$ as
\begin{equation} \label{eq:2.mpb}
\phi(x+N) = \phi(x) + Nm, \quad x \in \Z,
\end{equation}
for a fixed $m\in \R$.  We denote $H_N(\phi)$ by $H_{N,m}(\phi)$ under
this boundary condition.

We consider the Langevin dynamics $\phi_t= (\phi_t(x))_{x\in \T_N} \in 
\mathcal{Y}_N$ associated with $H$ adding an asymmetry:
\begin{align}  \label{eq:phi_t}
d\phi_t(x)= pV'(\phi_t(x+1)-\phi_t(x)) dt
- q V'(\phi_t(x)-\phi_t(x-1)) dt  + dB_t(x), \;\, x \in \T_N,
\end{align}
with the modified periodic boundary condition
\begin{equation*}
\phi_t(N+1) = \phi_t(1) + Nm \quad \text{and}\quad 
\phi_t(0) = \phi_t(N) - Nm,
\end{equation*}
where $(B_t(x))_{x\in \T_N}$ is a family of independent Brownian motions, and
$p=1/2+\ga, q= 1/2-\ga$ with the parameter $\ga\in \R$ which describes
the strength of the asymmetry.  

The unnormalizable measure 
$d\hat\mu_{N,m}  = e^{-H_{N,m}(\phi)}d\phi, d\phi=\prod_{x\in \T_N}d\phi(x)$, on 
$\mathcal{Y}_N$ is stationary for the process
$\phi_t$, determined by \eqref{eq:phi_t}, for every $\ga\in \R$ and symmetric
(reversible) if $\ga=0$; see Lemma \ref{lem:Rev+Stat-phi} and Remark
\ref{rem:2.1-A}, to make $\hat\mu_{N,m}$ normalizable, below.

\subsection{Tilt process $\eta_t$}

The tilt or slope variables, denoted by $\eta=(\eta(x))_{x\in \T_N} \in 
\mathcal{X}_N :=\R^N$, are defined from the height variables $\phi$ by 
\begin{align} \label{eq:eta-phi}
\eta(x):= \phi(x+1)-\phi(x) \; \big(\!\equiv \nabla\phi(x)\big), \quad x \in \T_N.
\end{align}
The modified periodic boundary condition \eqref{eq:2.mpb} leads to the
periodic boundary condition for $\eta$ extended to $\Z$
\begin{align} \label{eq:eta-p}
\eta(x+N) = \eta(x), \quad x \in \Z.
\end{align}
Then, the equation \eqref{eq:phi_t} 
for the height variables $\phi_t$ leads to
the equation for the tilt variables $\eta_t=(\eta_t(x))_{x\in \T_N}\in \mathcal{X}_N$:
\begin{align} \label{eq:eta_t-pert-2}
d\eta_t(x) =& p\Big\{ V'(\eta_t(x+1)) - V'(\eta_t(x)) \Big\}dt
+ q \Big\{ V'(\eta_t(x-1)) - V'(\eta_t(x)) \Big\} dt \\
& \qquad + dB_t(x+1) - dB_t(x),  \quad x \in \T_N,  \notag
\end{align}
with $B_t(N+1) = B_t(1)$ and the periodic boundary condition
\begin{equation*}
\eta_t(N+1) = \eta_t(1) \quad \text{and}\quad 
\eta_t(0) = \eta_t(N),
\end{equation*}
cf.\ \cite{DGP}, p.552 on $\Z$;  \cite{LNS} for a totally asymmetric case
($p=0, q=1 \; \text{i.e.}\ \ga= -1/2$); also \cite{F05}, (9.2) and  
\cite{GPV}  for a symmetric case (i.e.\ $\ga=0$).
Note that the $\eta$-process conserves the sum of $\eta_t(x)$ in $x\in \T_N$
and, from \eqref{eq:2.mpb} and \eqref{eq:eta-phi}, we have
\begin{align}  \label{eq:2.cons}
\sum_{x\in\T_N} \eta_t(x) = \sum_{x\in\T_N} \eta_0(x) \; = Nm,
\end{align}
for every $t\ge 0$; see Section \ref{sec:1.16.1}.

\subsection{Stationarity, symmetric part and carr\'e du champ}

Now, let us discuss the stationarity in more detail and summarize the related facts.

\subsubsection{The process $\phi_t$}
\label{sec:2.3.1}

The generator at the $\phi$-level is easier to calculate, and this covers the
$\eta$-level computation; see, for example, \cite{FS97} and Remarks
\ref{rem:1.3} and \ref{rem:4.1-A}.  The generator $L$ 
of the $\phi$-process, which is determined by
\eqref{eq:phi_t}, is given as the sum $L=L_0+L_1$ of the symmetric part $L_0$ and
the asymmetric part $L_1$ written as
\begin{align*} 
L_0 & = \frac12 \sum_{x\in \T_N} \Big(\partial_{\phi(x)}^2 + 
\big(V'(\phi(x+1)-\phi(x))- V'(\phi(x)-\phi(x-1)) \big) \partial_{\phi(x)} \Big)\\
& = \frac12 \sum_{x\in \T_N} \big( \partial_{\phi(x)}^2 - \partial_{\phi(x)} H \cdot \partial_{\phi(x)} \big)\\
L_1 & = \ga \sum_{x\in \T_N} \big(V'(\phi(x+1)-\phi(x)) + V'(\phi(x)-\phi(x-1)) \big) \partial_{\phi(x)},
\end{align*}
where $\partial_{\phi(x)} = \frac{\partial}{\partial\phi(x)}$ and
$H(\phi)$ is defined in \eqref{eq:H}.  
The operator $L_0$ is sometimes called the distorted Laplacian.
The carr\'e du champ operator acting on functions $F=F(\phi)$ on $\R^N$ is defined by
\begin{align} \label{eq:2.CC}
\Ga(F)(\phi) := LF^2(\phi)-2F(\phi) LF(\phi).
\end{align}
Note that the contribution of the first-order differential operators (i.e.\
the second term of $L_0$ and $L_1$) cancels out.  In particular, $\Ga(F)$ is
determined from the symmetric part $L_0$.

The operators $L_0$ and $L_1$ can also be written, by the summation by parts, as
\begin{align*} 
L_0 & = \frac12 \sum_{x\in \T_N} \Big(\partial_{\phi(x)}^2 + 
V'(\phi(x+1)-\phi(x)) \big(\partial_{\phi(x)} -\partial_{\phi(x+1)} \big) \Big) \\
L_1 & = \ga \sum_{x\in \T_N} V'(\phi(x+1)-\phi(x))
\big(\partial_{\phi(x)} +\partial_{\phi(x+1)} \big).
\end{align*}
Note that, by \eqref{eq:2.mpb},
$\partial_{\phi(N+1)}= \partial_{\phi(1)}$ and $\partial_{\phi(0)} = \partial_{\phi(N)}$.

Recall the measure $\hat\mu_{N,m}$ on $\mathcal{Y}_N=\R^N$
defined at the end of Section \ref{sec:2.1}.

\begin{lem}  \label{lem:Rev+Stat-phi}
For functions $F=F(\phi)$ and $G=G(\phi) \in C^2(\R^N)$, if all integrals
in the following converge, we have
\begin{align*} 
& \int_{\R^N} GL_0F d\hat\mu_{N,m} = \int_{\R^N} FL_0G d\hat\mu_{N,m}
= -\frac12 \int_{\R^N} \sum_{x\in \T_N} \partial_{\phi(x)} F \cdot \partial_{\phi(x)} G \, d\hat\mu_{N,m}, \\
& \int_{\R^N} G L_1 F d\hat\mu_{N,m} = -\int_{\R^N} F L_1 G d\hat\mu_{N,m},
\intertext{and}
& \Ga(F) = \sum_{x\in \T_N} \big(\partial_{\phi(x)} F\big)^2 =: |D_\phi F|^2,
\end{align*}
where $D_\phi F = (\partial_{\phi(x)} F)_{x\in \T_N}$.
\end{lem}

The proof is simple and straightforward, so it is omitted.

\begin{rem}  \label{rem:2.1-A}
We write $\lq\lq\phi\sim\phi'$'' when $\lq\lq\phi(x)-\phi'(x)=$ const'' holds
for two elements $\phi=(\phi(x))_{x\in \T_N}$ and $\phi'=(\phi'(x))_{x\in \T_N}$
of $\mathcal{Y}_N= \R^N$.  Then, $\lq\lq\sim$'' is an equivalence relation
on $\mathcal{Y}_N$, and one can consider the quotient space
$\widetilde{\mathcal{Y}}_N := \mathcal{Y}_N/\!\sim$.  The $\phi$-process
$\phi_t$ on $\mathcal{Y}_N$ has the following property.  If two
initial values satisfy $\lq\lq\phi_0\sim\phi_0'$'', then $\lq\lq\phi_t\sim\phi_t'$'' 
holds for $t>0$.  Therefore, the $\phi$-process is defined on the
quotient space $\widetilde{\mathcal{Y}}_N$.  A function $F$ on
$\widetilde{\mathcal{Y}}_N$ is a function on $\mathcal{Y}_N$ which satisfies
$F(\phi)=F(\phi')$ if $\phi\sim\phi'$.  The operators $L_0$ and $L_1$ are
accordingly considered on $\widetilde{\mathcal{Y}}_N$.
The Hamiltonian $H_{N,m}(\phi)$ has the above property so it is a function
on $\widetilde{\mathcal{Y}}_N$.

If $\eta=\nabla\phi$, $\eta'=\nabla\phi'$ on $\T_N$ and
$\phi\sim\phi'$, then we have $\eta=\eta'$ by \eqref{eq:eta-phi}.
This means that the $\phi$-process on $\widetilde{\mathcal{Y}}_N$ and 
the $\eta$-process on $\mathcal{X}_N = \R^N$ are essentially the same.
In this sense, the stationary measure $\hat\mu_{N,m}$ on $\mathcal{Y}_N$
is normalizable if we consider it on $\widetilde{\mathcal{Y}}_N$.
Recall that $H_{N,m}(\phi)$  is a function on $\widetilde{\mathcal{Y}}_N$.
We denote it by $\bar\mu_{N,m}$, which is essentially the same as
$\mu_{N,m}$ in \eqref{eq:2.13} below under the correspondence
$\eta=\nabla\phi$ on $\T_N$.
\end{rem}

\subsubsection{The process $\eta_t$}
\label{sec:2.3.2}

Here, we summarize the generator $L=L_0+L_1$ of the process $\eta_t$, which
is determined by \eqref{eq:eta_t-pert-2},  the Dirichlet form and the carr\'e du 
champ, which are determined from the symmetric part $L_0$.
Later in Section \ref{sec:4.2}, we will consider on a region $\La_\ell$ (instead of the 
whole lattice $\T_N$) with boundary conditions.

\begin{lem}  \label{lem:1.10}
The generator $L=L_0+L_1$ of the process $\eta_t$, determined by 
\eqref{eq:eta_t-pert-2}, is given by
\begin{align*} 
L_0 & = \sum_{x\in \T_N} \big( \partial_{\eta(x)}^2 - \partial_{\eta(x)}\partial_{\eta(x-1)}
\big)  + \frac12 \sum_{x\in \T_N} \big( V'(\eta(x+1)) -2V'(\eta(x))+ V'(\eta(x-1)) \big) \partial_{\eta(x)} \\
& = \frac12 \sum_{x\in \T_N} 
\Big( \big( \partial_{\eta(x)} - \partial_{\eta(x-1)} \big)^2
- \big( V'(\eta(x)) - V'(\eta(x-1)) \big)
\big( \partial_{\eta(x)} - \partial_{\eta(x-1)} \big)\Big),   \\
L_1 & = \ga \sum_{x\in \T_N} \big( V'(\eta(x+1)) - V'(\eta(x-1)) \big) \partial_{\eta(x)},
\end{align*}
where $\partial_{\eta(x)} = \frac{\partial}{\partial\eta(x)}$.
Note that $\eta(N+1)=\eta(1), \eta(0)= \eta(N)$ and $\partial_{\eta(0)}
= \partial_{\eta(N)}$ by \eqref{eq:eta-p}.
\end{lem}

\begin{proof}
For a function $F=F(\eta)$ on $\R^N$, by It\^o's formula,
\begin{align*} 
dF(\eta_t) = \sum_{x\in \T_N} \partial_{\eta(x)}F (\eta_t)\,d\eta_t(x)
+ \frac12 \sum_{x,y\in \T_N}  \partial_{\eta(x)} \partial_{\eta(y)} F(\eta_t)
\, d\eta_t(x)d\eta_t(y).
\end{align*} 
Here, by \eqref{eq:eta_t-pert-2}, we have that
\begin{align*} 
d\eta_t(x)d\eta_t(y) = \left\{
\begin{aligned}
2 dt,& \quad x=y, \\
-dt,& \quad x=y\pm 1, \\
0,& \quad \text{otherwise}.
\end{aligned}  \right.
\end{align*} 
Thus, we get $L=L_0+L_1$ again by using \eqref{eq:eta_t-pert-2}
for $d\eta_t(x)$ in the first term of $dF(\eta_t)$.
\end{proof}

\begin{rem}  \label{rem:1.3}
For $F=F(\eta)$ on $\mathcal{X}_N = \R^N$, one can calculate
\begin{equation}  \label{eq:d-phi}
\partial_{\phi(x)}F(\eta) = \sum_y \partial_{\eta(y)}F(\eta) \cdot \frac{\partial\eta(y)}{\partial\phi(x)}
= \partial_{\eta(x-1)}F(\eta) -  \partial_{\eta(x)}F(\eta).
\end{equation}
Thus, the operator $L$ for $\phi$ given in Section \ref{sec:2.3.1} can be rewritten
as in Lemma \ref{lem:1.10} for $\eta$.
The identity \eqref{eq:d-phi} can be rephrased as
$$
D_\phi = \nabla^* D_\eta,
$$
where $D_\phi = (\partial_{\phi(x)})_{x\in \T_N}$,
$D_\eta = (\partial_{\eta(x)})_{x\in \T_N}$, $\nabla g(x)  := g(x+1)-g(x)$ and 
$\nabla^* g(x) := g(x-1)-g(x)$ for $g=(g(x))_{x\in \T_N}$.
\end{rem}

Let $d\mu = \frac1Z e^{-H(\eta)} d\eta$ be a probability measure on
$\R^N$, where 
$$
H(\eta) = \sum_{x\in \T_N} V(\eta(x)) \quad \text{ and } \quad
d\eta = \prod_{x\in \T_N} d\eta(x),
$$
for $\eta= (\eta(x))_{x\in \T_N}$.  We assume 
\begin{align} \label{eq:2.V}
V(\zeta)\ge c_1 \zeta^2-c_2, \quad c_1>0, \, c_2\in \R,
\end{align}
then $\mu$ is normalizable, i.e.\ $Z<\infty$.
One can actually define $\mu_u$ parametrized by its mean as follows.  First,
for $\la\in \R$, define a probability measure $\hat\nu_\la$ on $\R$ by
\begin{align*}
&\hat{\nu}_\la(d\zeta) = \frac1{\hat Z_\la} e^{-V(\zeta)+\la\zeta}d\zeta, 
\quad \zeta\in \R.
\end{align*}
By \eqref{eq:2.V},
\begin{align} \label{eq:2.Z}
\hat Z_\la :=\int_\R e^{-V(\zeta)+\la\zeta}d\zeta <\infty
\end{align}
for all $\la\in \R$.   Then, we set
\begin{align}  \notag
& u\equiv u(\la) := E^{\hat\nu_\la}[\zeta],
\intertext{denoting the mean of $\zeta$ under $\hat \nu_\la$, and we have}
& u'(\la) = E^{\hat\nu_\la}[(\zeta- E^{\hat\nu_\la}[\zeta])^2] >0.
\label{eq:u-deri}
\end{align}
Thus, $u=u(\la)$ has an inverse function $\la=\la (u)$, actually given by
$\la(u)=\si'(u) = E^{\hat\nu_\la}[V'(\eta)]$, where $\si(u)$ is the surface
tension defined as in (5.2) of \cite{F05}.
We set $\nu_u := \hat\nu_{\la(u)}$
for $u\in \R$.  We denote the product measure of $\nu_u$ on
$\mathcal{X}_N=\R^N$ by 
\begin{align} \label{eq:GC}
\mu_u := \nu_u^{\otimes\T_N},
\end{align}
called the grand canonical measure.  The measure $\mu$ corresponds to the case
of $\la=0$.  Then, $\{\mu_u\}_{u\in\R}$ are reversible for $\eta_t$,
determined by \eqref{eq:eta_t-pert-2},
when $\ga=0$, and stationary for general $\ga\in \R$, as in the following lemma.
See also \cite{F05}, Section 5.5 for more details.

\begin{lem}  \label{lem:Rev+Stat-eta}
Let $u\in \R$. 
For functions $F=F(\eta)$ and $G=G(\eta) \in C^2(\R^N)$, if all integrals
in the following converge, we have
\begin{align*} 
 \int_{\R^N} GL_0F d\mu_u & = \int_{\R^N} FL_0G d\mu_u  \\
& = -\frac12 \int_{\R^N} \sum_{x\in \T_N} \big( \partial_{\eta(x)} -\partial_{\eta(x-1)} \big) F \cdot
\big( \partial_{\eta(x)} -\partial_{\eta(x-1)} \big) G\, d\mu_u, \\
\int_{\R^N} G L_1 F d\mu_u &= -\int_{\R^N} F L_1 G d\mu_u.
\end{align*}
For the carr\'e du champ of the $\eta$-process, we have
\begin{align} \label{eq:2.13-Q}
\Ga(F)  = \sum_{x\in \T_N} \big( \partial_{\eta(x)} F -\partial_{\eta(x-1)} F \big)^2 
= \sum_{x\in \T_N} \big( \nabla^*D_\eta F(x) \big)^2 
=: |\nabla_{\T_N}^* D_\eta F|^2,
\end{align}
recall Remark \ref{rem:1.3} for $\nabla^*$ and $D_\eta$. 
\end{lem}

Note that $\nabla^*$ in the expression of $\Ga(F)$ plays a role of 
Riemannian metric.  This lemma is shown by a direct calculation for
$\mu_u$.  It also follows from Lemma 
\ref{lem:Rev+Stat-phi} under the interpretation \eqref{eq:d-phi} in
Remark \ref{rem:1.3}.  The result for $\mu_u$ 
follows also from that for $\mu$ ($=\hat\nu_0^{\otimes \T_N}$,
i.e.\ the case of $\la=0$) and the 
conservation law \eqref{eq:2.cons}; see the next Section \ref{sec:1.16.1}.  
The proof of this lemma is omitted.

\subsubsection{Conservation law and canonical measures}  \label{sec:1.16.1}

The $\eta$-process, as defined by \eqref{eq:eta_t-pert-2},
has a conservation law \eqref{eq:2.cons} so we need to consider 
the $\eta$-process and also the operators $L_0, L_1$, as given in 
Lemma \ref{lem:1.10}, on the state space
\begin{align*}
\mathcal{X}_{N,m} := \{\eta\in \R^N; \eta^{(N)}=m\}
\end{align*}
for each $m\in \R$, where
\begin{align*}
\eta^{(N)} := \frac1N \sum_{x\in\T_N} \eta(x).
\end{align*}

Consider the conditional probabilities
\begin{align} \label{eq:2.13}
\mu_{N,m} := \mu_u(\cdot|\eta^{(N)}=m) \equiv \nu_u^{\otimes \T_N}(\cdot|\eta^{(N)}=m),
\end{align}
which are called canonical measures.  Note that $\mu_{N,m}$ does not depend 
on the choice of $u\in \R$ or $\la\in \R$, since
\begin{align*}
\hat\nu_\la^{\otimes \T_N} = \frac1{\hat Z_\la^N} e^{-\sum_{x\in\T_N} V(\eta(x))
+ \la N \eta^{(N)}} \prod_{x\in \T_N} d\eta(x)
\end{align*}
and $\eta^{(N)}=m$ is a constant under the conditioning.

The generator $L_0$ is symmetric with respect to $\mu_{N,m}$ for every $m\in \R$;
cf.\ the proof of Lemma \ref{lem:1.14} below.

\section{Boltzmann-Gibbs principle in the $L^p$ sense} \label{sec:BG}

\subsection{Main result}

Let $\eta_t = (\eta_t(x))_{x\in \T_N}$ be the solution of \eqref{eq:eta_t-pert-2}
and consider its diffusive scaling in time:
\begin{align} \label{eq:3.Diff}
\eta_t^N := \eta_{N^2t}.
\end{align}

We state our main result as Theorem \ref{BG-2-B}.  When $p=2$,
Theorem \ref{BG-2-B}  is equivalent, including the order of the constants in 
$T, N, \ell$, to the known results obtained in \cite{BFS}, Theorem 4.1,
though the models are different.  We extend these results to the
$L^p$ setting.  Our estimates are uniform in $\ga\in \R$.

For a function $f=f(\eta)$ on $\mathcal{X}_N=\R^N$, we denote
\begin{align}  \label{eq:EA}
\tilde{f}(u) \equiv \lan f\ran(u) :=E^{\mu_u}[f],
\end{align}
the expectation taken under the grand canonical measure $\mu_u$,
which is defined in
\eqref{eq:GC}, for $u\in \R$.  Notation $E_{\mu_u}[\,\cdot\,]$ indicates that the
initial distribution of the process $\eta^N_t$ is $\mu_u$.  

The following theorem provides the precise statement for \eqref{eq:1.BG}
of second- and first-orders.  We take $u_0\in \R$ and fix it.
For $\ell\in \N, \ell<N/2$, set $\La_\ell:=[-\ell,\ell-1]\cap\Z$.
The sample average of $\eta=(\eta(y))_{y\in \La_\ell+x}$ is defined by
\begin{align}\label{eq:eta-ell-B}
\eta^{(\ell)}(x) := \frac1{2\ell} \sum_{y=-\ell}^{\ell-1} \eta(x+y). 
\end{align}
This corresponds to \eqref{eq:1.1} in the present setting, in particular, when $d=1$.

\begin{thm}   \label{BG-2-B}
Assume
\begin{align} \label{eq:Assump}
V\in C^2(\R) \quad \text{and} \quad 0<c_-\le V''\le c_+<\infty,
\end{align}
for some positive constants $c_+$ and $c_-$.
Let $\ell_0 \ge 1$, $p\ge 4$, and let $f$ be a local $L^{p'}(\mu_{u_0})$ function 
on $\mathcal{X}_N=\R^N$ for some $p'>p$
depending only on $\eta_{\La_{\ell_0}}=(\eta(x))_{x\in \La_{\ell_0}}$
such that $ \tilde{f}(u_0)  =0$ and 
$\tilde{f}' (u_0)=0$.  Then, there exists a constant $C=C(u_0,\ell_0)>0$, 
which is uniform in $N, \ell, \ga\in \R$, such that, for $T>0$, 
$\ell\geq \ell_0$ and a function $h$ on $[0,T]\times \T_N$, we have
\begin{align}  \label{eq:3.1-A}
& E_{\mu_{u_0}}\bigg[ \; \sup_{0\leq t\leq T}\Big| \int_0^t ds \; \sum_{x\in\T_N} 
   h(s,x)  r_s^{N,\ell;2}(x)  \Big|^p\; \bigg]\\
& \hskip 3mm \notag
    \le  C \Big( T^{(p-2)/2} N^{-p/2}\ell^{p/2} +T^{p-1} N^p \ell^{-3p/2}\Big)
        \|f\|^p_{L^{p'}(\mu_{u_0})} \int_0^T \frac{1}{N}\sum_{x\in{\T_N}}|h(t,x)|^p dt,
\end{align}
where 
\begin{align}  \label{eq:3.2-A}
r_s^{N,\ell;2}(x)=f(\tau_x \eta^{N}_s)  - \frac{1}{2} \tilde{f}''(u_0)
\Big\{\Big(\big(\eta^{N}_s\big)^{(\ell)}(x)-u_0\Big)^2    
  -\frac{{\rm var}(u_0)}{2\ell+1} \Big\}.
\end{align}
and 
${\rm var}(u)=u'(\la(u)) = E^{\mu_u}[(\eta(0)-u)^2]$ is the variance of $\mu_u$;
see \eqref{eq:u-deri}.
Recall \eqref{eq:eta-ell-B} for $\eta^{(\ell)}(x)$ and we take $\eta_s^N$ 
for $\eta$.
On the other hand, when only $\tilde{f}(u_0)=0$ is known, we have
\begin{align}  \label{eq:3.3-A}
& E_{\mu_{u_0}}\bigg[\; \sup_{0\leq t\leq T}\Big| \int_0^t  ds \; \sum_{x\in{\T_N}} 
  \; h(s,x) r_s^{N,\ell;1}(x)\Big|^p\; \bigg]\\
&\hskip 3mm \notag
  \le  C \Big( T^{(p-2)/2} N^{-p/2}\ell^{p} +T^{p-1} N^p \ell^{-p}\Big)
        \|f\|^p_{L^{p'}(\mu_{u_0})} \int_0^T \frac{1}{N}\sum_{x\in{\T_N}}|h(t,x)|^p dt,
\end{align}
where
\begin{align}  \label{eq:3.4-A}
r_s^{N,\ell;1}(x) = f(\tau_x \eta^{N}_s) 
  - \tilde{f}'(u_0)\Big\{\big(\eta^{N}_s\big)^{(\ell)}(x) - u_0\Big\}.
\end{align}
\end{thm}

These estimates are called the second-order and the first-order 
Boltzmann-Gibbs principles, respectively.  Note that the right-hand side of
\eqref{eq:3.3-A} is $\ell^{p/2}$ times the right-hand side of \eqref{eq:3.1-A}.
This means that a higher-order Taylor expansion provides a better bound.
As we noted, to discuss the KPZ scaling limit, we consider in the weakly asymmetric
regime choosing $\ga N^{-1/2}$ instead of $\ga$ and this is covered by 
Theorem  \ref{BG-2-B}.  However, it also covers the strongly asymmetric regime.

We need the conditions \eqref{eq:4.Vconv}, (SG) and \eqref{eq:4.VSconv} for $V$,
as stated below.  If the condition \eqref{eq:Assump} holds, then these three conditions
and also \eqref{eq:2.V} are satisfied.  Moreover, under the condition
\eqref{eq:Assump}, $V'$ is globally Lipschitz continuous; therefore,
the SDEs \eqref{eq:phi_t} for $\phi_t$ and \eqref{eq:eta_t-pert-2} for $\eta_t$ 
and their localized versions have unique global-in-time strong solutions.

\begin{rem}
When applying Theorem  \ref{BG-2-B} to the KPZ scaling limit, 
the first-order Boltzmann-Gibbs principle \eqref{eq:3.3-A} is used for
the term with another scaling factor $N^{-1/2}$ inside the integral
on the left-hand side (cf.\ \cite{BFS}), resulting in an additional factor of
$N^{-p/2}$ on the right-hand side.
In \eqref{eq:3.1-A}, if we take $\ell=N^{3/4}$, the front factor on the 
right-hand side becomes
$$
T^{(p-2)/2}N^{-p/8} + T^{p-1}N^{-p/8}.
$$
On the other hand, the estimate \eqref{eq:3.3-A} multiplied by $N^{-p/2}$,
as explained above, gives the front factor on the right-hand side
$$
T^{(p-2)/2}N^{-p}\ell^p + T^{p-1}N^{p/2}\ell^{-p}
= T^{(p-2)/2}N^{-p/4} + T^{p-1}N^{-p/4}
$$
taking $\ell=N^{3/4}$ again.  The power $p$ is natural, since we are considering the 
$p$th moment on the left-hand side.
\end{rem}

\subsection{Outline of the proof of Theorem \ref{BG-2-B}}
\label{sec:3.1}

Before giving the details, here we outline the proof of Theorem \ref{BG-2-B};
in particular, we explain how the Littlewood-Paley-Stein inequality will be used.
We consider in $\eta$-variables under the canonical measure $\mu_{N,m}$ 
on the canonical 
space $\mathcal{X}_{N,m}$ with $m \equiv m_N= Nu_0$, instead of the 
grand canonical measure $\mu_{u_0}$ on the space
$\mathcal{X}_N$, so that the process $\eta_t$ is ergodic.

First, we show It\^o-Tanaka trick (Corollary \ref{cor:1.7}) in the $L^p$ sense
under equilibrium:
\begin{align}  \label{eq:1.IT}
E_{\mu_{N,m}}&\Big[ \sup_{t\in [0,T]} \Big| \int_0^t V(s,\eta_s^N) ds\Big|^p \Big]\\
&\le C_p N^{-p} T^{(p-2)/2} \int_0^TE^{\mu_{N,m}}\Big[ \{\Ga((-L_0)^{-1}V(t,\cdot))\}^{p/2} \Big] dt,
\notag
\end{align}
for $V(s,\eta)$ on $[0,T]\times\mathcal{X}_{N,m}$ satisfying 
$E^{\mu_{N,m}}[V(s,\cdot)]=0$ for every $s\in [0,T]$, where the
constant $C_p>0$ is uniform in $N$, $m\in \R$ and $\ga\in \R$.
Note that $(-L_0)^{-1}V(t,\cdot)$ exists if $V$ satisfies 
$E^{\mu_{N,m}}[V(t,\cdot)]=0$ at least under the condition (SG) for $-L_0$ 
on $\mathcal{X}_{N,m}$; see Section \ref{sec:1.20.4}.

Second, we note the Littlewood-Paley-Stein inequality at the $\eta$-level;
see Proposition \ref{prop:LPS} below for detail.  Let $(\mathcal{X}, \mu,\Theta)$
be $(\mathcal{X}_{N,m},\mu_{N,m},\T_N)$ or its localized version
$(\mathcal{X}_{\ell,m},\mu_{\ell,m},\La_\ell)$. Then,  
for every $1<p<\infty$ and a function $f$ on $\mathcal{X}$,
\begin{align} \label{eq:LPS-2}
c \|\nabla_\Theta^* D_\eta f\|_{L^p(\mu)} \le \|(-L_0)^{1/2}f\|_{L^p(\mu)} 
\le C \|\nabla_\Theta^* D_\eta f\|_{L^p(\mu)},
\end{align}
holds with constants $c=c_{p}, C=C_{p}>0$, which are uniform in $N, m$ 
and in $\ell, m$ under localization.  Here, $\nabla_\Theta^*g$ denotes
the vector $(\nabla^* g(x))_{x: \lan x-1,x\ran\subset \Theta}$, $L_0$ is 
defined at the $\eta$-level and it is taken as $L_{0,\ell}$ for the localized version
on $\mathcal{X}_{\ell,m}$.  When $p=2$, \eqref{eq:LPS-2} holds as equality
so that $c=C=1$.  A weaker estimate is given in Theorem \ref{thm:LPS} 
below, but it is not sufficient for our purposes.

By the formula for the carr\'e du champ $\Ga$ given in Lemma 
\ref{lem:Rev+Stat-eta} and then by the lower bound in the 
Littlewood-Paley-Stein inequality \eqref{eq:LPS-2} on $\T_N$, we have
\begin{align}  \label{eq:1.28}
E^{\mu_{N,m}}\Big[ \{\Ga((-L_0)^{-1}V)\}^{p/2} \Big]
& = E^{\mu_{N,m}}\Big[ \big| \nabla_{\T_N}^* D_\eta(-L_0)^{-1}V\big|^{p} \Big] \\
& \le c^{-p}  \big\| (-L_0)^{1/2}(-L_0)^{-1}V \big\|_{L^p(\mu_{N,m})}^p   \notag  \\
& = c^{-p}  \big\| (-L_0)^{-1/2}V \big\|_{L^p(\mu_{N,m})}^p.  \notag
\end{align}

\begin{rem}
A weaker estimate with different $p$ and $q$ on both sides of \eqref{eq:1.28},
as in \cite{Shi02}, is sufficient for the following arguments.
\end{rem}

To estimate the right-hand side of \eqref{eq:1.28}, we apply the
$L^p$ variational formula (Lemma \ref{lem:1.11}).  This avoids 
working directly with the nonlocal operator $(-L_0)^{-1/2}$.  Then, 
by localizing our argument in the region of size $\ell \, (<\!\!< N)$,
we apply the weak $L^p$ Poincar\'e inequality (Lemma \ref{lem:1.15})
and the localized version of the Littlewood-Paley-Stein inequality
\eqref{eq:LPS-2} on $\La_\ell$; see the proof of Proposition \ref{prop:1.21}.
See Remark \ref{rem:1.7} below for the reason that 
the localization is necessary.

Afterward, based on the bounds obtained in this way, we essentially follow
the line of proof in \cite{GoJ} and \cite{BFS}, extending it to the
$L^p$ setting.  Specifically, by properly
taking the function $V$ in \eqref{eq:1.IT}, we show the one-block estimate
(Lemma \ref{lem:1.22}, using \eqref{eq:LPS-2} on $\T_N$), 
the iteration lemma (Lemma \ref{globalrenormalization}), 
the two-blocks estimate (Lemma \ref{globaltwo-blocks})
and the further estimate (Lemma \ref{EE_1block}) by establishing
the equivalence of ensemble in the $L^p$ sense (Proposition \ref{thm:EE}).  
The proof of Theorem \ref{BG-2-B} is concluded by combining all these estimates.

\begin{rem}  \label{rem:1.7}
Without the localization argument (i.e.\ taking $\ell=N$ in the above), we have
$$
\big\| (-L_0)^{-1/2}V \big\|_{L^p(\mu_{N,m})} \le C_{p,q} N \|  V \|_{L^q(\mu_{N,m})}, 
\quad 2<p<q<\infty.
$$
Therefore, by \eqref{eq:1.IT} and  \eqref{eq:1.28}, we obtain
\begin{align}  \label{eq:1.30}
E_{\mu_{N,m}}\Big[ \sup_{t\in [0,T]} \Big| \int_0^t V(s,\eta_s^N) ds\Big|^p \Big]
\le C_{p,q}  T^{(p-2)/2} \int_0^T 
  \big\| V(t,\cdot) \big\|_{L^q(\mu_{N,m})}^p dt,
\end{align}
for $V$ satisfying $E^{\mu_{N,m}}[V(t,\cdot)]=0$ for every $t\in [0,T]$
and $2<p<q<\infty$.  However, when we apply this to the KPZ scaling
limit, $V= O(N)$ (the sum of $N$ terms) and this estimate turns out to be too rough.
\end{rem}

\section{Several tools for the proof of Theorem \ref{BG-2-B}}
\label{sec:4}

This section introduces the tools necessary for proving Theorem 
\ref{BG-2-B}.  Sections \ref{sec:4.2} through \ref{sec:VF} deal with the symmetric part
$L_0$ of the generator only.  In particular, all estimates given in this section are 
uniform in $\ga\in\R$.

\subsection{It\^o-Tanaka trick in the $L^p$ sense in equilibrium} 

We first state It\^o-Tanaka trick, also known as Kipnis-Varadhan estimate, in the $L^p$
setting.  Let $L=L_0+L_1$ be the generator  of the process $\eta_t$ 
(see Lemma \ref{lem:1.10})
with symmetric part $L_0=(L+L^*)/2$ and asymmetric part $L_1=(L-L^*)/2$,
where $L^*$ denotes the adjoint of $L$ with respect to $\mu_u$ defined in
\eqref{eq:GC}; recall Lemma \ref{lem:Rev+Stat-eta} which shows $L_0^*=L_0$ and
$L_1^*=-L_1$.  Let 
\begin{align}  \label{eq:4.CC}
\Ga(F,G)(\eta)= L_0(FG)(\eta)-GL_0F(\eta)-FL_0G(\eta)
\end{align}
be the corresponding
carr\'e du champ; recall \eqref{eq:2.CC} for $\Ga(F)= \Ga(F,F)$.
Note that, since the asymmetric part $L_1$ is 
a first-order differential operator, $L_1(FG)-GL_1F-FL_1G=0$.

Recalling Section \ref{sec:1.16.1} for the notation $\mathcal{X}_{N,m}$, $\mu_{N,m}$,
and  $E_{\mu_{N,m}}$, $E^{\mu_{N,m}}$ explained above Theorem \ref{BG-2-B}
(replacing $\mu_u$ by $\mu_{N,m}$),
we have the following; cf.\ \cite{GJ} Lemma 2 for $p\ge 1$, \cite{KLO} Lemma 2.4
for $p=2$.  This type of estimate is also known in non-equilibrium setting;
see, for example, \cite{FUY} Lemma 3.3.

\begin{lem}  \label{lem:1.6}
Let $p\ge 2$.
For a function $F=F(t,\eta)$ on $[0,T]\times \mathcal{X}_{N,m}$, we have
\begin{align}  \label{eq:4.1-U}
E_{\mu_{N,m}}\Big[ \sup_{t\in [0,T]} \Big| \int_0^t L_0F(s,\eta_s) ds\Big|^p \Big]
\le C_p T^{(p-2)/2} \int_0^T E^{\mu_{N,m}}\Big[ \{\Ga(F(t,\cdot))\}^{p/2} \Big]dt,
\end{align}
where $C_p>0$ is uniform in $N$, $m\in \R$ and $\ga\in\R$.
In particular, if we take $F$ such that $-L_0F(s,\eta)=V(s,\eta)$ from a given $V$ 
on $[0,T]\times \mathcal{X}_{N,m}$ satisfying $E^{\mu_{N,m}}[V(s,\cdot)]=0$
for each $s\in [0,T]$, then we have
\begin{align}   \label{eq:4.2-U}
E_{\mu_{N,m}} & \Big[ \sup_{t\in [0,T]} \Big| \int_0^t V(s,\eta_s) ds\Big|^p \Big] \\
& \le C_p T^{(p-2)/2} \int_0^TE^{\mu_{N,m}}
\Big[ \{\Ga((-L_0)^{-1}V(t,\cdot))\}^{p/2} \Big] dt.
\notag
\end{align}
See Section \ref{sec:3.1} below \eqref{eq:1.IT} for the existence of $(-L_0)^{-1}$
noting the condition (SG) in Section \ref{sec:1.20.4}.
\end{lem}

\begin{proof}
By Dynkin's formula, we have
$$
F(t,\eta_t) = F(0,\eta_0) + \int_0^t (\partial_s+L) F(s,\eta_s)ds + M_t^+,
$$
for $t\ge 0$,
where $(M_t^+)$ is a martingale with respect to the filtration $(\mathcal{F}_t)$
with $\mathcal{F}_t= \si\{\eta_s; 0\le s \le t\}$.
Similarly, for each $T>0$,  Dynkin's formula for the backward process provides
$$
F(T-t,\eta_{T-t}) = F(T,\eta_T) + \int_0^t (-\partial_s+L_0-L_1) F(T-s,\eta_{T-s})ds + M_t^-,
$$
for $t\in [0,T]$,
where $(M_t^-)$ is a martingale with respect to the filtration $(\mathcal{F}_t^-)$
with $\mathcal{F}_t^-= \si\{\eta_{T-s}; 0\le s \le t\}$.
The quadratic variations of these two martingales are given by
\begin{equation} \label{eq:quad-M}
\begin{aligned} 
& \lan M^+\ran_t = \int_0^t \Ga(F(s,\cdot))(\eta_s)ds, \\
& \lan M^-\ran_t = \int_0^t \Ga(F(T-s,\cdot))(\eta_{T-s})ds.
\end{aligned}
\end{equation}
By these formulas, for $t\in [0,T]$, we obtain
$$
\int_0^t 2L_0 F(s,\eta_s)ds = - M_t^+ + M_{T-t}^- - M_T^-.
$$
We then apply Burkholder's inequality to see that the left-hand side of 
\eqref{eq:4.1-U} is bounded by
$$
C\Big\{ E_{\mu_{N,m}}\big[ \lan M^+\ran_T^{p/2} \big] 
+ E_{\mu_{N,m}}\big[ \lan M^-\ran_T^{p/2} \big]\Big\},
$$
where $C=C_p>0$ is uniform in $N$, $m$ and $\ga$.

Thus, by \eqref{eq:quad-M}, 
applying H\"older's inequality with $p'=p/2$ and $q'= p/(p-2)$
(the case $p=2$ is clear so we assume $p>2$)
and then noting the stationarity of $\eta_t$ under $\mu_{N,m}$, we have
\begin{align*}
E_{\mu_{N,m}}\big[ \lan M^+\ran_T^{p/2} \big] 
& = E_{\mu_{N,m}}\Big[ \Big\{ \int_0^T \Ga(F(t,\cdot))(\eta_t)dt \Big\}^{p/2} \Big] \\
& \le T^{(p-2)/p \,\cdot\, p/2}
E_{\mu_{N,m}}\Big[ \int_0^T \big\{ \Ga(F(t,\cdot))(\eta_t) \big\}^{p/2} dt \Big] \\
&= T^{(p-2)/2} \int_0^T E^{\mu_{N,m}}\big[\big\{ \Ga(F(t,\cdot))\big\}^{p/2}\big] dt.
\end{align*}
Similar for $E\big[ \lan M^-\ran_T^{p/2} \big]$, and we obtain \eqref{eq:4.1-U}.

The second estimate \eqref{eq:4.2-U} is immediate from \eqref{eq:4.1-U}.
\end{proof}

The estimates obtained in Lemma \ref{lem:1.6} can easily be restated for the
diffusively scaled process $\eta_t^N = \eta_{N^2 t}$, given in \eqref{eq:3.Diff},
as follows.

\begin{cor}  \label{cor:1.7}
Let $p\ge 2$.
For a function $F=F(t,\eta)$ on $[0,T]\times \mathcal{X}_{N,m}$, we have
\begin{align*}
E_{\mu_{N,m}}\Big[ \sup_{t\in [0,T]} \Big| \int_0^t L_0F(s,\eta_s^N) ds\Big|^p \Big]
\le C_p N^{-p} T^{(p-2)/2} \int_0^T E^{\mu_{N,m}}\Big[ \{\Ga(F(t,\cdot))\}^{p/2} \Big] dt.
\end{align*}
In particular, for  $V(s,\eta)$ on $[0,T]\times \mathcal{X}_{N,m}$ satisfying
$E^{\mu_{N,m}}[V(s,\cdot)]=0$ for each $s\in [0,T]$,
\begin{align*}
E_{\mu_{N,m}}\Big[ \sup_{t\in [0,T]} \Big| \int_0^t V(s,\eta_s^N) ds\Big|^p \Big]
\le C_p N^{-p} T^{(p-2)/2} \int_0^T E^{\mu_{N,m}}\Big[ \{\Ga((-L_0)^{-1}V(t,\cdot))\}^{p/2} \Big] dt.
\end{align*}
For  $V(t,s,\eta)$, $s\in [0,t]$  satisfying
$E^{\mu_{N,m}}[V(t,s,\cdot)]=0$, we have without supremum in $t$
\begin{align*}
E_{\mu_{N,m}}\Big[ \Big| \int_0^t V(t,s,\eta_s^N) ds\Big|^p \Big]
\le C_p N^{-p} t^{(p-2)/2} \int_0^t E^{\mu_{N,m}}\Big[ \{\Ga((-L_0)^{-1}V(t,s,\cdot))\}^{p/2} \Big] ds.
\end{align*}
\end{cor}

\begin{proof}
By the change of variables $u=N^2 s$, we have
\begin{align*}
\int_0^t L_0F(s,\eta_s^N) ds = N^{-2} \int_0^{N^2t} L_0F(N^{-2}u ,\eta_u) du.
\end{align*}
Therefore, by the first bound \eqref{eq:4.1-U} of Lemma \ref{lem:1.6}, we obtain
\begin{align*}
E_{\mu_{N,m}}\Big[ \sup_{t\in [0,T]} \Big| \int_0^t L_0F(s,\eta_s^N) ds\Big|^p \Big]
 & \le N^{-2p} C_p (N^2 T)^{(p-2)/2} \int_0^{N^2T} E^{\mu_{N,m}}\Big[ \{\Ga(F(N^{-2}u,\cdot))\}^{p/2} \Big] du\\
 & = N^{-2p} C_p (N^2 T)^{(p-2)/2} N^2 \int_0^T E^{\mu_{N,m}}\Big[ \{\Ga(F(t,\cdot))\}^{p/2} \Big] dt,
\end{align*}
which shows the first bound of the corollary.  The second follows from the first.
For the third, the left-hand side is bounded by $E_{\mu_{N,m}}[ 
\sup_{s\in [0,t]} | \int_0^s V(t,r,\eta_r^N) dr|^p ]$ for which 
the above estimate is applicable by taking $T=t$.
\end{proof}

\subsection{Localization of symmetric $\eta$-process and its
reversible measures}
\label{sec:4.2}

\subsubsection{The $\phi$-process}

First, let us consider the localization of the symmetric $\phi$-process
introduced in Section \ref{sec:2.1}.  More precisely, for $\ell\in \N, \ell<N/2$,
we define $(\phi_t(x))_{|x|\le \ell-1}$ by \eqref{eq:phi_t} with $\ga= 0$
and the Dirichlet boundary conditions $\phi_t(-\ell)=0$ and $\phi_t(\ell)=2\ell m$
for some $m\in \R$.  Note that we considered the modified periodic
boundary condition on $\T_N$.

The generator $L_{0,\ell}$ of this $\phi$-process is the operator $L_0$ given
in Section \ref{sec:2.3.1} with the sum restricted to ``$-\ell+1\le x \le \ell-1$'', i.e.\
\begin{align} \label{eq:4.Lphi}
L_{0,\ell} = \frac12 \sum_{x=-\ell+1}^{\ell-1} \Big( \partial_{\phi(x)}^2
+ \big( V'(\phi(x+1)-\phi(x)) -V'(\phi(x)-\phi(x-1)) \big) \partial_{\phi(x)} \Big)
\end{align}
with the boundary
conditions $\phi(-\ell)=0$ and $\phi(\ell) = 2\ell m$.  The corresponding
carr\'e du champ is given by $\Ga(F)$ in Lemma \ref{lem:Rev+Stat-phi}
with the sum restricted to ``$-\ell+1\le x \le \ell-1$''.

It is convenient to introduce the state space for the localized $\phi$-process
with the boundary conditions:
\begin{align} \label{eq:4.Y}
\mathcal{Y}_{\ell,m} = \big\{\phi= (\phi(x))_{x=-\ell}^{\ell};
\phi(-\ell)=0, \phi(\ell) = 2\ell m \big\} \cong \R^{2\ell-1}.
\end{align}

The probability measure $\bar\mu_{\ell,m}$ (now normalizable) on $\mathcal{Y}_{\ell,m}$
defined by
\begin{align} \label{eq:4.bmu}
\bar\mu_{\ell,m}(d\phi) = \frac1{Z_{\ell,m}} e^{-H_{\ell,m}(\phi)} \prod_{x=-\ell+1}
^{\ell-1}d\phi(x),
\end{align}
is symmetric for $L_{0,\ell}$, where
\begin{align} \label{eq:Hem}
H_{\ell,m}(\phi) = \sum_{x=-\ell}^{\ell-1} V(\phi(x+1)-\phi(x))
\end{align}
with $\phi(-\ell)=0$ and $\phi(\ell) = 2\ell m$, and $Z_{\ell,m}$ is a
normalizing constant.

\subsubsection{The $\eta$-process}

Next, determine $(\eta_t(x))_{x\in \La_\ell}, \La_\ell :=[-\ell,\ell-1] \cap\Z$,
(we will omit $\lq\lq\cap\Z$'' in the following) by \eqref{eq:eta-phi} 
(for $x\in \La_\ell$) from
the above $\phi$-process.  Then, they satisfy the equation \eqref{eq:eta_t-pert-2}
with $p=q=1/2$ (i.e.\ $\ga=0$) for $x\in [-\ell+1,\ell-2]$.
At two boundary points $x=-\ell$ and $\ell-1$, we have
\begin{align}  \label{eq:4.bdry}
\begin{aligned}
d\eta_t(-\ell) &= d \phi_t(-\ell+1) - d\phi_t(-\ell) =  d\phi_t(-\ell+1) \\
& = \frac12\Big\{V'(\eta_t(-\ell+1)) - V'(\eta_t(-\ell)) \Big\}dt + dB_t(-\ell+1), 
 \\
d\eta_t(\ell-1) &= d \phi_t(\ell) - d\phi_t(\ell-1) = - d\phi_t(\ell-1) 
  \\
& = -\frac12\Big\{V'(\eta_t(\ell-1)) - V'(\eta_t(\ell-2)) \Big\}dt - dB_t(\ell-1).
\end{aligned}
\end{align}
In other words, we may take $V'(\eta_t(-\ell-1)) = V'(\eta_t(-\ell))$ and 
$B_t(-\ell)=0$ in \eqref{eq:eta_t-pert-2} for $d\eta_t(-\ell)$, and 
$V'(\eta_t(\ell)) = V'(\eta_t(\ell-1))$ and
$B_t(\ell)=0$ in \eqref{eq:eta_t-pert-2}  for $d\eta_t(\ell-1)$.  
From these equations or the construction of $(\eta_t(x))_{x\in \La_\ell}$,
we have the conservation law:
\begin{align*}
\sum_{x=-\ell}^{\ell-1} \eta_t(x) = \sum_{x=-\ell}^{\ell-1}\eta_0(x) = 2\ell m.
\end{align*}

The state space of the $\eta$-process on $\La_\ell$ with the conservation 
law is given by
\begin{align*}
\mathcal{X}_{\ell,m} := \big\{\eta =(\eta(x))_{x\in \La_\ell}\in \R^{2\ell}; \eta^{(\ell)}=m\big\} \cong \R^{2\ell-1}.
\end{align*}
where $\eta^{(\ell)} =\eta^{(\ell)}(0)$ is defined by \eqref{eq:eta-ell-B} taking
$x=0$, and consider the conditional probabilities on $\mathcal{X}_{\ell,m}$:
\begin{align}  \label{mu-ellm}
\mu_{\ell,m} := \nu_u^{\otimes 2\ell}(\cdot|\eta^{(\ell)}=m).
\end{align}

The localization of the operator $L_0\equiv L_{0,\ell}$ on the region 
$\La_\ell$ is naturally
defined as a generator of the process $(\eta_t(x))_{x \in \La_\ell}$
introduced as above.  This is given by the following lemma.

\begin{lem}  \label{lem:1.13}
The operator $L_0\equiv L_{0,\ell}$ is given by
\begin{align*}
L_{0,\ell} = \frac12 \sum_{x=-\ell+1}^{\ell-1}
\Big( \big( \partial_{\eta(x)} - \partial_{\eta(x-1)} \big)^2
- \big( V'(\eta(x)) - V'(\eta(x-1)) \big)\big( \partial_{\eta(x)} - \partial_{\eta(x-1)} \big)  \Big).
\end{align*}
\end{lem}

Note that the pair $\lan x-1,x\ran$ for $\partial_{\eta(x)} - \partial_{\eta(x-1)}$
in the above sum appears only for the inside of the region: 
$\lan x-1,x\ran \subset \La_\ell=[-\ell,\ell-1]$
and no boundary term appears.  Also note that $L_{0,\ell}$ acts on the space
$\mathcal{X}_{\ell,m}$.  This is seen also from \eqref{eq:1.vanish} below.
As in Remark \ref{rem:1.3} \eqref{eq:d-phi} (and also \eqref{eq:4.12-A} below), 
this $L_{0,\ell}$ is obtained from 
that in \eqref{eq:4.Lphi} by replacing $\partial_{\phi(x)}$ with
$\partial_{\eta(x-1)} -\partial_{\eta(x)}$.

\begin{proof}
As in the proof of Lemma \ref{lem:1.10}, for a function $F=F(\eta)$ on 
$\R^{2\ell}$, by It\^o's formula,
\begin{align*} 
dF(\eta_t) = \sum_{x=-\ell}^{\ell-1} \partial_{\eta(x)}F (\eta_t)\,d\eta_t(x)
+ \frac12 \sum_{x,y=-\ell}^{\ell-1}  \partial_{\eta(x)} \partial_{\eta(y)} F(\eta_t)
\, d\eta_t(x)d\eta_t(y).
\end{align*} 
Here, we have that
\begin{align*} 
d\eta_t(x)d\eta_t(y) = \left\{
\begin{aligned}
2 dt,& \quad x=y \in [-\ell+1,\ell-2], \\
dt,& \quad x=y=-\ell \text{ or }  \ell-1, \\
-dt,& \quad x=y\pm 1, \;\;  x,y\in [-\ell,\ell-1], \\
0,& \quad \text{otherwise}.
\end{aligned}  \right.
\end{align*} 
Thus, we can calculate the generator $L_{0,\ell}$ by using \eqref{eq:eta_t-pert-2}
with $p=q=1/2$ for $d\eta_t(x)$, $x\in [-\ell+1,\ell-2]$, and \eqref{eq:4.bdry}
for $d\eta_t(-\ell)$ and $d\eta_t(\ell-1)$.
\end{proof}

The operator $L_{0,\ell}$ is symmetric with respect to $\mu_{\ell,m}$ 
for every $m\in \R$,
and the corresponding Dirichlet form is calculated as in the next lemma.

\begin{lem}  \label{lem:1.14}
For functions $F=F(\eta)$ and $G=G(\eta) \in C^2(\R^{2\ell})$, 
if all integrals in the following converge, we have
\begin{align*} 
\int_{\R^{2\ell}} GL_{0,\ell}F d\nu_u^{\otimes 2\ell} 
& = \int_{\R^{2\ell}} FL_{0,\ell}G d\nu_u^{\otimes 2\ell}   \\
& = -\frac12 \int_{\R^{2\ell}} \sum_{x=-\ell+1}^{\ell-1}
 \big( \partial_{\eta(x)} -\partial_{\eta(x-1)} \big) F \cdot
\big( \partial_{\eta(x)} -\partial_{\eta(x-1)} \big) G\, d\nu_u^{\otimes 2\ell}.
\end{align*}
Moreover, for every $m\in \R$ and functions $F$ and $G\in C^2(\mathcal{X}_{\ell,m})$,
if all integrals in the following converge, we have
\begin{align*} 
\int_{\mathcal{X}_{\ell,m}} GL_{0,\ell}F d\mu_{\ell,m} & 
= \int_{\mathcal{X}_{\ell,m}} FL_{0,\ell}G d\mu_{\ell,m}   \\
& = -\frac12 \int_{\mathcal{X}_{\ell,m}} \sum_{x=-\ell+1}^{\ell-1}
 \big( \partial_{\eta(x)} -\partial_{\eta(x-1)} \big) F \cdot
\big( \partial_{\eta(x)} -\partial_{\eta(x-1)} \big) G\, d\mu_{\ell,m} \\
& =: - \mathcal{D}_{\ell,m}(F,G).
\end{align*}
Note again that only the pair $\lan x-1,x\ran \subset \La_\ell=[-\ell,\ell-1]$ appears
in the sum.
\end{lem}

\begin{proof}
The first identity in the lemma is shown by a direct calculation noting that
$$
d\nu_u^{\otimes 2\ell} = \frac1{\hat Z_{\la(u)}^{2\ell}}
\exp \Big\{ -\sum_{x=-\ell}^{\ell-1} V(\eta(x))+\la(u) 2\ell \eta^{(\ell)} \Big\}
\prod_{x=-\ell}^{\ell-1} d\eta(x).
$$

To show the second, take a function $F(\eta) =f(\eta^{(\ell)}) \tilde F(\eta)$ 
on $\R^{2\ell}$ given as a skew product of functions $f$ on $\R$ and 
$\tilde F$ on $\mathcal{X}_{\ell, m}$.  Then, since
\begin{align}  \label{eq:1.vanish}
\big\{ \partial_{\eta(x)} - \partial_{\eta(x-1)} \big\} f(\eta^{(\ell)})
= f'(\eta^{(\ell)}) \big\{ \partial_{\eta(x)} - \partial_{\eta(x-1)} \big\} \eta^{(\ell)}
=0,
\end{align}
for $x\in [-\ell+1,\ell-1]$,
we see $\big\{ \partial_{\eta(x)} - \partial_{\eta(x-1)} \big\} F(\eta)
= f(\eta^{(\ell)})\big\{ \partial_{\eta(x)} - \partial_{\eta(x-1)} \big\} \tilde F(\eta)$
and $L_0 F(\eta) = f(\eta^{(\ell)}) L_0 \tilde F(\eta)$.
In particular, from the first identity, we obtain
\begin{align*}
\int_\R & f(m) \nu_u^{\otimes 2\ell} \circ (\eta^{(\ell)})^{-1}(dm)
\int_{\mathcal{X}_{\ell,m}} \tilde GL_{0,\ell} \tilde F d\mu_{\ell,m} \\
& = -\frac12 \int_\R f(m) \nu_u^{\otimes 2\ell} \circ (\eta^{(\ell)})^{-1}(dm) \\
& \hskip 20mm \times 
\int_{\mathcal{X}_{\ell,m}} \sum_{x\in [-\ell+1,\ell-1]}
 \big( \partial_{\eta(x)} -\partial_{\eta(x-1)} \big) \tilde F \cdot
\big( \partial_{\eta(x)} -\partial_{\eta(x-1)} \big) \tilde G\, d\mu_{\ell,m},
\end{align*}
where $\nu_u^{\otimes 2\ell} \circ (\eta^{(\ell)})^{-1}$ denotes the 
distribution of $\eta^{(\ell)}$ under $\nu_u^{\otimes 2\ell}$.
Since $f(m)$ is arbitrary, we obtain the second identity.
\end{proof}

Note that, if $\phi$ is $\bar\mu_{\ell,m}$-distributed (recall \eqref{eq:4.bmu}
for $\bar\mu_{\ell,m}$), then $\eta=(\eta(x))_{x\in \La_\ell}$,
defined by \eqref{eq:eta-phi} for $x\in \La_\ell$ from $\phi\in
\mathcal{Y}_{\ell,m}$, is $\mu_{\ell,m}$-distributed.

\begin{rem}  \label{rem:4.1-A}
The $\phi$-variable $\phi=(\phi(x))_{x=-\ell}^\ell \in \mathcal{Y}_{\ell,m}$
can be identified with $\eta=(\eta(x))_{x\in \La_\ell} \in \mathcal{X}_{\ell,m}$
by $\eta(x) = \nabla\phi(x) := \phi(x+1)-\phi(x)$
defined by \eqref{eq:eta-phi} for $x\in \La_\ell$, or
$\phi(x) = \sum_{y=-\ell}^{x-1}\eta(y), x \in [-\ell,\ell]$.
Thus, a function $F=F(\phi)$ on $\mathcal{Y}_{\ell,m}$ is identified with
$F=F(\eta)$ on $\mathcal{X}_{\ell,m}$.  Then, as in Remark \ref{rem:1.3}, we have
\begin{equation}  \label{eq:4.12-A}
\partial_{\phi(x)}F(\eta) = \partial_{\eta(x-1)}F(\eta)  -  \partial_{\eta(x)}F(\eta),
\quad x \in [-\ell+1,\ell-1].
\end{equation}
\end{rem}

\subsection{Sobolev-type inequality due to Littlewood-Paley-Stein inequality}
\label{sec:4.3}

We state the Sobolev-type inequality, which follows from the 
Littlewood-Paley-Stein inequality, shown by Bakry \cite{Ba87}.  
As previously mentioned, we refer to it as
the Littlewood-Paley-Stein inequality.
Recall $\ell\in \N, \ell<N/2$, $\La_\ell =[-\ell,\ell-1]$ and Remark \ref{rem:1.3} for
$$
\nabla^* g(x) = g(x-1)-g(x),\quad x \in [-\ell+1,\ell-1],
$$ 
for $g=(g(x))_{x\in \La_\ell}$.

\begin{prop} \label{prop:LPS}
Let $1<p<\infty$ and assume the function $V$ satisfies the condition
\begin{align} \label{eq:4.Vconv}
V\in C^2(\R) \quad \text{and} \quad V''\ge 0.
\end{align}
\noindent
{\rm (1)}  Then, for every $f=f(\eta)$ on $\mathcal{X}_{\ell,m}$, we have
\begin{align} \label{eq:LPS-D}
c \|\nabla_{\La_\ell}^* D_\eta f\|_{L^p(\mu_{\ell,m})}  
\le \|(-L_{0,\ell})^{1/2}f\|_{L^p(\mu_{\ell,m})} \le 
C \|\nabla_{\La_\ell}^* D_\eta f\|_{L^p(\mu_{\ell,m})},
\end{align}
where $D_\eta f(\eta) = (D_\eta f(x,\eta))_{x\in \La_\ell}$ with 
$D_\eta f(x,\eta):= \partial_{\eta(x)} f(\eta)$ and 
$\|\nabla_{\La_\ell}^* D_\eta f\|_{L^p(\mu_{\ell,m})}$ denotes the
$L^p(\mu_{\ell,m})$-norm of the Euclidean norm of
\begin{align} \label{eq:4.15-A}
\nabla_{\La_\ell}^* D_\eta f (\eta):= 
\big( \nabla^* D_\eta f(x,\eta)\big)_{x=-\ell+1}^{\ell-1}
\in \R^{2\ell-1},
\end{align}
that is,
\begin{align} \label{eq:4.15-B}
|\nabla_{\La_\ell}^* D_\eta f(\eta)| := \Big( \sum_{x=-\ell+1}^{\ell-1}
(\nabla^* D_\eta f(x,\eta))^2\Big)^{1/2}.
\end{align}
Here, $c=c_p,C=C_p>0$ are universal constants, depending
only on $p$ and, in particular, being uniform in $\ell$ and $m$.  
Note that $-L_{0,\ell}$ is symmetric on $L^2(\mu_{\ell,m})$ and 
the operator $(-L_{0,\ell})^{1/2}$ is defined via its spectral decomposition.

\noindent
{\rm (2)} The same estimate as \eqref{eq:LPS-D} holds for
$(\mathcal{X}_{N,m}, \mu_{N,m},\T_N,L_0)$ replacing
$(\mathcal{X}_{\ell,m}, \mu_{\ell,m},\La_\ell,L_{0,\ell})$.
Note that $|\nabla_{\T_N}^* D_\eta f|^2$ is defined as in \eqref{eq:2.13-Q}.
\end{prop}

\begin{proof}
(1) We consider in terms of the $\phi$-variables.
To apply for functions $f$ of the variable $\eta$, we may
replace $D_\phi$ by $\nabla^* D_\eta$ according to  \eqref{eq:4.12-A}.

We apply the results of \cite{Ba87}, taking 
$E=\mathcal{Y}_{\ell,m}\cong \R^{2\ell-1}$ and $\rho(\phi) = 
e^{-H_{\ell,m}(\phi)}/ Z_{\ell,m}$ with $H_{\ell,m}(\phi)$ given in 
\eqref{eq:Hem} and $\Z_{\ell,m}$ below that.  Since the Ricci curvature tensor
of our space $E$ is $0$, the symmetric tensor 
$$
R:= Ric-\nabla\nabla (\log \rho)
= (\partial_{\phi(x)}\partial_{\phi(y)} H_{\ell,m}(\phi))_{x,y\in [-\ell+1,\ell-1]}
$$
is the Hessian of $H_{\ell,m}(\phi)$.  By a simple calculation, 
for $x,y\in [-\ell+1,\ell-1]$, we see that
\begin{align*}
\partial_{\phi(x)}\partial_{\phi(y)} H_{\ell,m}(\phi)
= \left\{
\begin{aligned}
V''(\phi(x+1)-\phi(x)) + V''(\phi(x)-& \phi(x-1)),
\quad x =y, \\
-V''(\phi(x+1)-\phi(x)), & 
\qquad y=x+1, \\
-V''(\phi(x)-\phi(x-1)), & 
\qquad y=x-1, \\
0,& \qquad |x-y|\ge 2.
\end{aligned}
\right.
\end{align*}
Therefore, for $\xi=(\xi_x)_{x=-\ell+1}^{\ell-1}\in \R^{2\ell-1}$,
by setting $\xi_{-\ell}=\xi_\ell=0$, we have
\begin{align*}
(R\xi,\xi) & = \sum_{x=-\ell+1}^{\ell-1}
\{V''(\phi(x+1)-\phi(x)) + V''(\phi(x)-\phi(x-1))\} \xi_x^2 \\
& \quad -2 \sum_{x=-\ell+1}^{\ell-2}
V''(\phi(x+1)-\phi(x)) \xi_x \xi_{x+1} \\
& = \sum_{x=-\ell}^{\ell-1}
V''(\phi(x+1)-\phi(x)) (\xi_{x+1}- \xi_x)^2,
\end{align*}
where $(R\xi,\xi)$ denotes the inner product of $R\xi$ and $\xi$ in $\R^{2\ell-1}$.
If $V$ satisfies \eqref{eq:4.Vconv}, then we have $(R\xi,\xi)\ge 0$.  Therefore, 
we can take $\a=0$ in \cite{Ba87}, p.138 and obtain 
\begin{align} \label{eq:LPS-D-B}
c \|D_\phi f\|_{L^p(\bar\mu_{\ell,m})}  
\le \|(-L_{0,\ell})^{1/2}f\|_{L^p(\bar\mu_{\ell,m})} \le 
C \|D_\phi f\|_{L^p(\bar\mu_{\ell,m})}.
\end{align}
for every $f=f(\phi)$ on $\mathcal{Y}_{\ell,m}$, with universal constants
$c=c_p,C=C_p>0$ depending only on $p$, where $D_\phi f = (\partial_{\phi(x)}
f)_{x=-\ell+1}^{\ell-1}$.  This implies \eqref{eq:LPS-D}.

(2) 
The estimate \eqref{eq:LPS-D-B} holds with the same constants $c_p, C_p>0$
at the $\phi$-level on the quotient space $\widetilde{\mathcal{Y}}_{N} = 
\mathcal{Y}_{N}/\!\sim$ with a normalized measure $\bar \mu_{N,m}$;
see Remark \ref{rem:2.1-A} for $\widetilde{\mathcal{Y}}_{N}$ and
$\bar\mu_{N,m}$.
The corresponding $\eta$-space is $\mathcal{X}_{N,m}$ and we obtain 
the estimate \eqref{eq:LPS-D} on $\T_N$.
\end{proof}

\begin{rem}  \label{rem:4.2-C}
We can also apply Theorems 2 and 7 of \cite{Ba}, noting that 
$\Ga(u) = |Du|^2\equiv
\sum_{x=-\ell+1}^{\ell-1}(\partial_{\phi(x)} u)^2$ and the operator
$C=(-L_{0,\ell})^{1/2}$.  See the bottom of p.146 of \cite{Ba}
for the equivalence between $\Ga_2\ge 0$ and Hess $H_{\ell,m}\ge 0$ 
(as a matrix), which follows from the condition \eqref{eq:4.Vconv} 
as we saw above.
The result of \cite{Ba} implies the upper bound in \eqref{eq:LPS-D}
for every $1<p<\infty$; however, the lower bound is shown only for $p>2$.
\end{rem}

\subsection{Weak $L^p$ Poincar\'e inequality}
\label{sec:4.4}

\subsubsection{Spectral gap for $L_{0,\ell}$} \label{sec:1.20.4}

In the following, we discuss in general under the following spectral gap assumption:
\begin{itemize}
\item[(SG)] There exists a constant $C>0$, uniform in $\ell$ and $m$, such that
$$
\mathcal{D}_{\ell,m}(F,F)
\ge C\ell^{-2} \|F\|_{L^2(\mu_{\ell,m})}^2
$$
for every $F\in L^2(\mu_{\ell,m})$ satisfying $E^{\mu_{\ell,m}}[F]=0$;
recall Lemma \ref{lem:1.14} for $\mathcal{D}_{\ell,m}(F,F) \equiv
E^{\mu_{\ell,m}}[F\cdot (-L_{0,\ell})F]$.
\end{itemize}

In fact, Caputo \cite{Ca03} Corollary 5.2 shows that (SG) holds with a uniform
constant $C>0$, independent of $\ell$ and $m$, under the assumption for $V$
that $V=V_1+V_2$, where $V_1\in C^2(\R)$ satisfies $c_-\le V_1''(\eta) \le c_+$
for some $c_+\ge c_->0$, and $V_2\in C_b^2(\R)$; see also \cite{Ca04}.
The corresponding logarithmic Sobolev inequality is studied by  
\cite{MO}.  Theorem 6.1.14 in \cite{DS}  (p.242) with $\a >0$ and $\b=0$ 
(i.e.\ logarithmic Sobolev inequality) implies
(6.1.19) in Corollary 6.1.17 (p.244) and, from this corollary, we have the
lower bound of the spectral gap of $-L_{0,\ell}$ on $L^2(\mu_{\ell,m})$
given as $\lq\lq \ge 2/\a \ge 2/(c_2\ell^2)$'',
i.e.\ (SG) holds.

\begin{rem}
If $\mathcal{D}_{\ell,m}(F,F)=0$, then we have $\partial_{\eta(x)}F
= \partial_{\eta(x+1)}F$ for the function $F$ on $\mathcal{X}_{\ell,m}$.
This implies that $F\equiv$ const.  The condition {\rm (SG)} guarantees
the existence of $(-L_{0,\ell})^{-1}$ on $L^2(\mu_{\ell,m})\ominus\{\text{const}\}$.
The uniformity of the constant 
$C>0$ in $m$ in (SG) may be relaxed as in (2.5) of
\cite{BFS} in an averaged sense.
\end{rem}

\subsubsection{Weak $L^p$ Poincar\'e inequality}  \label{sec:1.22}

We have the following weak $L^p$ Poincar\'e inequality, written in terms
of the $\phi$-variables.  Recall \eqref{eq:4.bmu} for $\bar\mu_{\ell,m}$
and $D_\phi f = (\partial_{\phi(x)}f)_{x=-\ell+1}^{\ell-1}$ for a function 
$f=f(\phi)$ on $\mathcal{Y}_{\ell,m} = \R^{2\ell-1}$.

\begin{lem}  \label{lem:1.15}
{\rm (1)} We assume the condition {\rm (SG)} in Section \ref{sec:1.20.4}
and the following bound for $\bar\mu_{\ell,0}$ and every $p'>2$
\begin{align} \label{eq:uniform-phi}
\inf_{\bar\phi\in \R^{2\ell-1}} \int_{\R^{2\ell-1}} | \phi(x)-\bar\phi(x)|^{p'} 
\bar\mu_{\ell,0}(d\phi)   \le C \ell^{p'/2},
\end{align}
with $C=C_{p'}>0$, uniform in $x$ and $\ell$.  
Then, for $2< q < p<\infty$, we have
\begin{align}  \label{eq:4.WP}
\|f\|_{L^q(\bar\mu_{\ell,m})} \le C_{p,q} \ell \|D_\phi f\|_{L^p(\bar\mu_{\ell,m})}
\end{align}
for $f$ satisfying $\int_{\R^{2\ell-1}}f(\phi) \bar\mu_{\ell,m}(d\phi)=0$,
where $C_{p,q}>0$ is uniform in $\ell$ and $m$. \\
{\rm (2)} The uniform bound \eqref{eq:uniform-phi} holds if $V$ satisfies
\begin{align} \label{eq:4.VSconv}
V\in C^2(\R) \quad \text{and} \quad V''\ge c >0.
\end{align}
\end{lem}

\begin{proof}
(1) By the condition (SG) written in terms of the $\phi$-variables recalling
Remark \ref{rem:4.1-A}, we have
\begin{align}  \label{eq:1.P1}
\|f\|_{L^2(\bar\mu_{\ell,m})} \le C \ell \|D_\phi f\|_{L^2(\bar\mu_{\ell,m})}
\end{align}
for $f$ satisfying $\int_{\R^{2\ell-1}}f(\phi) \bar\mu_{\ell,m}(d\phi)=0$.

Next, assuming \eqref{eq:uniform-phi}, let us prove for every $p'>2$,
\begin{align}  \label{eq:1.P2}
\|f\|_{L^{p'}(\bar\mu_{\ell,m})} \le C_{p'} \ell \|D_\phi f\|_{L^\infty(\bar\mu_{\ell,m})}
\end{align}
for $f$ satisfying $\int_{\R^{2\ell-1}}f(\phi) \bar\mu_{\ell,m}(d\phi)=0$, where
$$
\|D_\phi f\|_{L^\infty} \equiv \|D_\phi f\|_{L^\infty(\bar\mu_{\ell,m})}
:= \sup_{\phi\in \R^{2\ell-1}} |D_\phi f(\phi)|
= \sup_{\phi\in \R^{2\ell-1}} \Big( \sum_{x=-\ell+1}^{\ell-1} \big| \partial_{\phi(x)}f(\phi)\big|^2\Big)^{1/2}.
$$
Indeed,
\begin{align*}
|f(\phi)| & = \Bigg| f(\phi)- \int_{\R^{2\ell-1}}f(\phi') \bar\mu_{\ell,m}(d\phi')
\Bigg| \\
& \le  \int_{\R^{2\ell-1}} |  f(\phi)-f(\phi') | \bar\mu_{\ell,m}(d\phi')  \\
& \le \|D_\phi f\|_{L^\infty} \int_{\R^{2\ell-1}} | \phi-\phi'| \bar\mu_{\ell,m}(d\phi'),
\end{align*}
where $| \phi-\phi'| = \big(\sum_{x=-\ell+1}^{\ell-1} | \phi(x)-\phi'(x)|^2\big)^{1/2}$.
Therefore,
\begin{align} \label{eq:1.28-A}
\|f\|_{L^{p'}(\bar\mu_{\ell,m})} & \le \|D_\phi f\|_{L^\infty}
\Big\{ \int_{\R^{2\ell-1}} \bar\mu_{\ell,m}(d\phi) \Big(\int_{\R^{2\ell-1}} | \phi-\phi'| 
\bar\mu_{\ell,m}(d\phi') \Big)^{p'} \Big\}^{1/p'} \\
& \le \|D_\phi f\|_{L^\infty}
\Big\{ \int_{\R^{2\ell-1}} \int_{\R^{2\ell-1}} | \phi-\phi'|^{p'} \bar\mu_{\ell,m}(d\phi)
\bar\mu_{\ell,m}(d\phi') \Big\}^{1/p'}   \notag\\
&\le C_{p'}\ell \|D_\phi f\|_{L^\infty}.  \notag
\end{align}

For the last inequality in \eqref{eq:1.28-A},
note that $C_{p'}$ does not depend on $m$, since one can first reduce to the
$0$-boundary measures $\bar\mu_{\ell,0}(d\phi) \bar\mu_{\ell,0}(d\phi')$
by simple shifts $\tilde \phi(x) := \phi(x)-m(x+\ell)$ and
$\tilde \phi'(x) := \phi'(x)-m(x+\ell)$, under which $\tilde\phi(x)-\tilde\phi'(x)
= \phi(x)-\phi'(x)$ holds.  Then, to show the last inequality in 
\eqref{eq:1.28-A}, note also
\begin{align*}
| \phi-\phi'|^{p'} &\le c_{p'}(| \phi -\bar\phi|^{p'} + | \phi'-\bar\phi|^{p'})
\intertext{for arbitrary $\bar\phi\in \R^{2\ell-1}$ and}
| \phi-\bar\phi|^{p'} & = \Big(\sum_{x=-\ell+1}^{\ell-1} |\phi(x)-\bar\phi(x)|^2\Big)^{p'/2}
\le \Big(\sum_{x=-\ell+1}^{\ell-1} 1 \Big)^{p'/2q'} \Big(\sum_{x=-\ell+1}^{\ell-1} |\phi(x)-\bar\phi(x)|^{p'}\Big)\\
& = (2\ell-1)^{p'/2q'} \sum_{x=-\ell+1}^{\ell-1} |\phi(x)-\bar\phi(x)|^{p'},
\end{align*}
where $2/{p'}+1/{q'}=1$ noting $p'>2$. 
Therefore, by the uniform bound \eqref{eq:uniform-phi} we assumed, we obtain
\begin{align*}
& \Big\{ \int_{\R^{2\ell-1}} \int_{\R^{2\ell-1}} | \phi-\phi'|^{p'} \bar\mu_{\ell,0}(d\phi)
\bar\mu_{\ell,0}(d\phi')\Big\}^{1/p'} \\
& \quad \le
\Big\{ 2c_{p'}
\inf_{\bar\phi\in \R^{2\ell-1}} \int_{\R^{2\ell-1}} | \phi-\bar\phi|^{p'} \bar\mu_{\ell,0}(d\phi) \Big\}^{1/p'} \\
& \quad \le C(\ell^{p'/2q'}\cdot \ell \cdot \ell^{p'/2}) ^{1/p'} = C \ell.
\end{align*}
This shows the last inequality in \eqref{eq:1.28-A} and, therefore, \eqref{eq:1.P2}
for some $C_{p'}>0$, uniform in $\ell$ and $m$.

Now we interpolate the two estimates \eqref{eq:1.P1} and \eqref{eq:1.P2}.
Let $T$ be an integral operator $D_\phi f \mapsto f$.
More precisely, let $L_*^p \equiv (L^p(\bar\mu_{\ell,m}))^{2\ell-1}_*$
be the class of all $u=(u_x(\phi))_{x\in [-\ell+1,\ell-1]}
\in  (L^p(\bar\mu_{\ell,m}))^{2\ell-1}$, which satisfy the consistency condition
$\partial_{\phi(x)}u_y = \partial_{\phi(y)}u_x$ in the distributions' sense for 
any $x,y \in [-\ell+1,\ell-1]$.  Then, recalling \eqref{eq:4.Y} for the state 
space $\mathcal{Y}_{\ell,m}$ of the $\phi$-variables,  at least for $u\in L_*^p \cap 
\big( C_b^\infty (\mathcal{Y}_{\ell,m})\big)^{2\ell-1}$,
since $u=\sum_{x=-\ell+1}^{\ell-1} u_x(\phi)d\phi(x)$ is a closed $1$-form (i.e.\ $du
= \sum_{x,y} \partial_{\phi(y)}u_x(\phi)\, d\phi(y) \wedge d\phi(x)=0$), by Poincar\'e's lemma,
it is an exact differential form written as $u=dg$ with a function $g(\phi) \equiv g(\phi;u) \in 
L^q(\bar\mu_{\ell,m}) \cap C_b^\infty (\mathcal{Y}_{\ell,m})$ defined by
\begin{align*}
g(\phi) := \int_0^1 \sum_{x=-\ell+1}^{\ell-1} u_x(\phi_\la)\partial_\la \phi_\la(x) d\la,
\quad \phi \in \mathcal{Y}_{\ell,m}= \R^{2\ell-1},
\end{align*}
where $\{\phi_\la\}_{\la\in [0,1]}$ is a smooth path in $\mathcal{Y}_{\ell,m}$ connecting
$\phi$ and a fixed configuration $\bar\phi = (\bar\phi(x)=mx+m\ell)_{x\in [-\ell+1,\ell-1]}
\in \mathcal{Y}_{\ell,m}$, in particular, $\phi_1=\phi$ and $\phi_0=\bar\phi$.  
Poincar\'e's lemma guarantees that $g(\phi)$ is determined independently of the choice 
of the path $\{\phi_\la\}$.  Then, set
\begin{align*}
f(\phi) \equiv f(\phi;u) := g(\phi) - \int_{\R^{2\ell-1}} g(\phi') \bar\mu_{\ell,m}(d\phi')
\end{align*}
so that $f$ satisfies
\begin{align*}
\int_{\R^{2\ell-1}} f(\phi) \bar\mu_{\ell,m}(d\phi)=0.
\end{align*}
This determines a linear map:
\begin{align*}
T: u \in L_*^p \cap \big( C_b^\infty (\mathcal{Y}_{\ell,m})\big)^{2\ell-1}\mapsto f 
\in L^q(\bar\mu_{\ell,m}) \cap C_b^\infty (\mathcal{Y}_{\ell,m})
\end{align*}
and $D_\phi f=u$ holds, i.e.\ $\partial_{\phi(x)}f(\phi) = u_x(\phi)$ holds for every
$-\ell+1 \le x \le \ell-1$.

By \eqref{eq:1.P1} and \eqref{eq:1.P2}, the map $T$ is continuously extended to
$L^2_*$ and $L^{\infty}_*$, which satisfies for every $p_1>2$
\begin{align*}
& \| Tu\|_{L^2(\bar\mu_{\ell,m})} \le C \ell \| u\|_{L_*^2}, \\
& \| Tu\|_{L^{p_1}(\bar\mu_{\ell,m})} \le C_{p_1} \ell \| u\|_{L_*^\infty}.
\end{align*}
Thus, by Riesz-Thorin interpolation theorem (see \cite{BeL}), for $p,q$
determined from $\th \in (0,1)$ such that
\begin{align*}
\frac1p=\frac{1-\th}2, \quad \frac1q= \frac{1-\th}2 + \frac{\th}{p_1},
\end{align*}
we obtain
\begin{align*}
\| Tu\|_{L^{q}(\bar\mu_{\ell,m})} \le C_{p,q} \ell \| u\|_{L_*^p},
\end{align*}
where $C_{p,q}= C^{1-\th}C_{p_1}^\th$.  Note that $\th = 1-2/p \in (0,1)$ implies
$p\in (2,\infty)$ and $1/p<1/q$ implies $q<p$.  Since
$p_1 >2$ can be taken arbitrarily large, one can choose any $q \,(<p)$ close to $p$.
This concludes the proof of (1) in the lemma.

(2) Under $\bar\mu_{\ell,0}$, if $V$ 
satisfies \eqref{eq:4.VSconv} (i.e.\ strict convexity), one can apply
Brascamp-Lieb inequality (cf.\ \cite{F05}, \cite{FT07} Theorem 3.1) recalling 
that \eqref{eq:uniform-phi} holds for pinned Brownian motion (with discrete time parameter $x$) taking $\bar\phi\equiv 0$.  This shows (2) in the lemma;
take $\bar\phi(x)= E^{\bar\mu_{\ell,0}}[\phi(x)]$.
\end{proof}

\begin{rem}  \label{rem:4.2}
{\rm (1)} If one can show \eqref{eq:1.P2} and \eqref{eq:1.28-A} with $p'=\infty$,
then \eqref{eq:4.WP} holds with $p=q$.  \\
{\rm (2)} Concerning the uniform bound \eqref{eq:uniform-phi}, choose
$\bar\phi(x) = E^{\bar\mu_{\ell,0}}[\phi(x)]$ and recall that
$\phi(x)= \sum_{y=-\ell+1}^x \eta(y)$.  If we consider under the measure with
$\phi(-\ell)=0$ and free for $\phi(\ell)$ instead of  $\bar\mu_{\ell,0}(d\phi)$,
then $\{\eta(y)\}$ is an i.i.d.\ sequence.  Thus, by the CLT,
$\phi(x)- \bar\phi(x)$  with $\bar\phi(x) = E[\phi(x)]$ behaves as 
$\sqrt{\ell} \times N(0,1)$ and 
we obtain  \eqref{eq:uniform-phi}; see \cite{F05} Remarks 8.1, 4.5 and \cite{GPV}.

We believe that \eqref{eq:uniform-phi} can be shown for much more general 
class of $V$. In fact, the estimate \eqref{eq:uniform-phi} with $p'=4$ follows 
from the estimate on $\th_n^*(t)$ taking $s=0$, so that $\th^*(s)=0$ and $q=0$ 
in the middle of p.429 of \cite{DH}, assuming only the existence of
exponential moments of $\hat \nu_0$.  We may extend this to every $p'>2$.
\end{rem}

\subsection{$L^p$ variational formula}  \label{sec:VF}

Since the operator $(-L_0)^{-1/2}$ in \eqref{eq:1.28} or \eqref{eq:1.V-B}
below is non-local, to localize the computation from $N$ to $\ell$, we show
the following $L^p$-$L^q$ variational formula.  We will use the upper bound;
see \eqref{eq:1.32} below.
In the following lemma, in general, $\mu$ is a probability measure on $\R^n$
(later, we will take $\mu=\mu_{N,m}, \R^n = \mathcal{X}_{N,m}$ with $n=N-1$ 
and others), $L_0$ is symmetric on $L^2(\mu)$ and has a spectral gap, i.e.,
there exists $c>0$ such that $E^\mu[F\cdot (-L_0)F] \ge c \| F\|_{L^2(\mu)}^2$
for every $F\in L^2(\mu)$ satisfying $E^\mu[F]=0$.
We use this property only for the existence of $(-L_0)^{-1/2}f$, and 
the estimate in the following lemma does not depend on $c$. 

\begin{lem} \label{lem:1.11}
Let $p, q>1$ satisfy $1/p+1/q=1$.  Then, 
for $f$ satisfying $E^\mu[f]=0$, we have
$$
\sup_{\fa\not\equiv \rm{const}} \frac{\int f\fa d\mu}{\|(-L_0)^{1/2}\fa\|_{L^q(\mu)}}
\le \|(-L_0)^{-1/2}f\|_{L^p(\mu)} 
\le 2 \sup_{\fa\not\equiv \rm{const}} \frac{\int f\fa d\mu}{\|(-L_0)^{1/2}\fa\|_{L^q(\mu)}}.
$$
\end{lem}

\begin{proof}
For simplicity, we write the norm $\|\cdot\|_{L^p(\mu)}$ as $\|\cdot\|_{L^p}$.
First, noting that $f= (-L_0)^{1/2}(-L_0)^{-1/2}f$ and that $(-L_0)^{-1/2}f$ exists
due to the spectral gap of $L_0$ and $E^\mu[f]=0$,
by the symmetry of $L_0$ and H\"older's inequality, we have
\begin{align*} 
\int f\fa d\mu & = \int (-L_0)^{-1/2}f (-L_0)^{1/2}\fa d\mu \\
& \le \|(-L_0)^{-1/2}f\|_{L^p} \| (-L_0)^{1/2}\fa\|_{L^q},
\end{align*}
for every $\fa$, which implies the first inequality (i.e.\ the lower bound).

To show  the second inequality (i.e.\ the upper bound), 
set $F= (-L_0)^{-1/2}f$.  Since $L^p(\mu)=(L^q(\mu))^*$, we have
$$
\|F\|_{L^p} =  \sup_{G\in L^q(\mu)} \frac{\int FG d\mu}{\|G\|_{L^q}},
$$
omitting to write $G\not=0$.  Note that we can replace $\int FG d\mu$ by
$|\int FG d\mu|$, by considering $\pm G$.  However, for 
$F\in L_0^p=\{F\in L^p(\mu); E^\mu[F]=0\}$ and decomposing $G=G_0+c$ with 
$G_0=G-c\in L_0^q$, $c= E^\mu[G]$, we have
\begin{align*} 
\|F\|_{L^p} & =\sup_{G\in L^q(\mu)} \frac{|\int FG d\mu|}{\|G\|_{L^q}}
= \sup_{G_0\in L_0^q(\mu)} \sup_{c\in \R} \frac{|\int FG_0 d\mu|}{\|G_0+c\|_{L^q}}  \\
&= \sup_{G_0\in L_0^q(\mu)} \frac{|\int FG_0 d\mu|}{\inf_{c\in \R}\|G_0+c\|_{L^q}} \\
& \le \frac1\k \sup_{G_0\in L_0^q(\mu)}  \frac{|\int FG_0 d\mu|}{\|G_0\|_{L^q}}
\end{align*}
since $E^\mu[F]=0$, where
\begin{align} \label{eq:1.kappa}
\k := \inf_{G_0\in L_0^q(\mu)} \inf_{c\in \R} \frac{ \|G_0+c\|_{L^q}}{ \|G_0\|_{L^q}}
\ge 0.
\end{align} 
Showing that $\k\ge 1/2$ completes the proof of the second inequality by taking 
$G_0= (-L_0)^{1/2}\fa$.

Finally, we show that $\k\ge 1/2$ for $\k$ in \eqref{eq:1.kappa}.  First note that
\begin{align*}
\k = \inf_{G\in L_0^q(\mu): \|G\|_{L^q}=1 } \inf_{c\in \R} \|G+c\|_{L^q},
\end{align*} 
by denoting $G_0/\|G_0\|_{L^q}$ by $G$.  Then, one can find a sequence 
$\{G_n\in L_0^q(\mu), c_n\in \R\}_n$ such that $\|G_n\|_{L^q}=1$
and $\|G_n+c_n\|_{L^q} \to \k$ as $n\to\infty$.  Note that $\{c_n\}$ is 
bounded, since $|c_n| \le \|G_n+c_n\|_{L^q} + \|G_n\|_{L^q}\to \k+1$ as
$n\to\infty$, so one can find its subsequence $c_{n'}$ and $c\in \R$
such that $c_{n'}\to c$ as $n'\to\infty$.  Since $\|G_{n'}\|_{L^q}=1$,
$\{G_{n'}\}$ is weakly compact in $L^q(\mu)$.  Thus, one can find 
its further subsequence
$\{G_{n''}\}$ such that $G_{n''}$ converges weakly to some $G\in L^q(\mu)$ as $n''\to\infty$.
However, by the weak convergence, $\lan G_{n''},1\ran \to \lan G,1\ran
=E^\mu[G]$ and this shows that
$G\in L_0^q(\mu)$.  The weak convergence also implies that
\begin{align*}
\|G+c\|_{L^q} \le \liminf_{n''\to\infty} \|G_{n''}+c\|_{L^q} 
 \le \liminf_{n''\to\infty} \big\{ \|G_{n''}+c_n''\|_{L^q} +|c-c_n''|\big\} =\k.
\end{align*} 
This together with $G\in L_0^q(\mu)$ shows 
$$
|c| = \Big|\int (G+c)d\mu\Big| \le \k.
$$
Therefore, taking the $\liminf$ as $n''\to 0$ in both sides of
$$
1= \|G_{n''}\|_{L^q} \le \|G_{n''} +c\|_{L^q} +|c|,
$$
we obtain $1\le 2\k$ for $\k$ in \eqref{eq:1.kappa}.
\end{proof}

\subsection{Equivalence of ensembles in the $L^p$ sense}

We prepare an $L^p$ version of the equivalence of ensembles.  We assume 
the existence of exponential moments: $E^{\mu_{u_0}}[e^{\la \eta(0)}]<\infty$ for
every $\la\in \R$.  This holds under the strict-convexity assumption 
\eqref{eq:4.VSconv} for $V$.

\begin{prop} (cf.\ Theorem 5.1 of \cite{BFS})\label{thm:EE}
Let $p\ge 4$ 
and let $f$ be a local $L^{p'}(\mu_{u_0})$ function on $\mathcal{X}_N = \R^N$
for some $p'>p$, depending only on $\eta_{\La_{\ell_0}} = 
(\eta(x))_{x\in \La_{\ell_0}}$, such that 
$\tilde f(u_0) = E^{\mu_{u_0}}[f]=0$ and $\tilde f'(u_0)=0$.  Then,
there exists a constant $C= C(u_0,\ell_0)$ (which is uniform in $u_0: |u_0|\le K$
for any $K>0$)
such that for $\ell\geq\ell_0$ we have
\begin{align}  \label{eq:4.EE-A}
\Big\| E^{\mu_{u_0}}\big[f| \eta^{(\ell)}\big] 
 - \frac12\Big\{(\eta^{(\ell)}-u_0)^2- \frac{{\rm var} (u_0) }{2\ell +1}\Big\} 
\tilde f''(u_0)
\Big\|_{L^p(\mu_{u_0})} 
\leq \ \frac{C\|f\|_{L^{p'}(\mu_{u_0})}}{\ell^{3/2}},
\end{align}
where ${\rm var} (u_0)$ is defined below \eqref{eq:3.2-A}.
On the other hand, when only $\tilde f(u_0)=0$ is known, we have
\begin{align}  \label{eq:4.EE-B}
&\Big\| E^{\mu_{u_0}}[f| \eta^{(\ell)}]  - (\eta^{(\ell)} - u_0)\tilde f'(u_0)  
\Big\|_{L^p(\mu_{u_0})} \ \leq \ \frac{C\|f\|_{L^{p'}(\mu_{u_0})}}{\ell}.
\end{align}
\end{prop}

The proof of Proposition \ref{thm:EE} is similar to the proof of Theorem 5.1 of 
\cite{BFS} or Proposition 5.1 of \cite{GJS}, though these treat the case 
of discrete variables.  The orders of the errors are roughly explained
as follows.  By the CLT for sum of i.i.d.\ under $\mu_{u_0}$,
$\eta^{(\ell)}-u_0 = \frac1{2\ell} \sum_{x=-\ell}^{\ell-1} (\eta(x)-u_0)
 =O(\ell^{-1/2})$.
The first estimate \eqref{eq:4.EE-A} cares approximation up to 
second-order terms so that the error is expected to be
$O(\ell^{-3/2})$.  The second estimate 
\eqref{eq:4.EE-B} cares up to
first-order term so that the error is expected to be $O(\ell^{-1})$.

More precisely, since we condition on $\eta^{(\ell)}$, on the set 
$1(\|y\|\le \de)$ for $y=\eta^{(\ell)}-u_0$, by applying the local CLT,
the error is estimated by $C\|f\|_{L^2(\mu_{u_0})} \ell^{-3/2}$; see the conclusion
of Step 7 in Section 5.2 of \cite{BFS} (note that $\|f\|_{L^2}$ is enough here, 
we take the power $1/4$ of both sides).  On the other hand,
on the set $1(\|y\|>\de)$, we apply H\"older's inequality (to separate 
$1(\|y\|>\de)$ and functions related to $f$).  Then, we apply classical 
Cram\'er's theorem for the large deviation estimate on $\mu_{u_0}(\| y\|>\de)$
for sum of i.i.d.\ random variables.  The error is estimated by 
$\|f\|_{L^{p'}(\mu_{u_0})} e^{-C\ell}$; see
Step 8 in Section 5.2 of \cite{BFS}.

The details are omitted.  For Proposition \ref{thm:EE}, $p\ge 2$ might be 
sufficient instead of $p\ge 4$.  Then, the same holds for Theorem 
\ref{BG-2-B}, Lemmas \ref{globaltwo-blocks} and \ref{EE_1block}.

\section{Proof of Theorem \ref{BG-2-B}}
\label{sec:5}

In this section, we always assume the condition \eqref{eq:Assump} given
in Theorem \ref{BG-2-B}.  This covers \eqref{eq:2.V}, \eqref{eq:4.Vconv},
\eqref{eq:4.VSconv} and Caputo's condition for (SG). 

As outlined in Section \ref{sec:3.1}, by Corollary \ref{cor:1.7} 
(It\^o-Tanaka trick) and the Littlewood-Paley-Stein inequality \eqref{eq:LPS-2},
and then estimating as in \eqref{eq:1.28}, we obtain
\begin{align} \label{eq:1.V-B}
E_{\mu_{N,m}}&\Big[ \sup_{t\in [0,T]} \Big| \int_0^t V(s,\eta_s^N) ds\Big|^p \Big]\\
& \le C_p N^{-p} T^{(p-2)/2} \int_0^T 
  \big\| (-L_0)^{-1/2}V(t,\cdot) \big\|_{L^p(\mu_{N,m})}^p dt,
  \notag
\end{align}
for $V$ satisfying $E^{\mu_{N,m}}[V(t,\cdot)]=0$ for every $t\in [0,T]$.
Note that $C_p$ is uniform in $N$, $m\in \R$ and $\ga\in \R$.

To estimate the right-hand side of \eqref{eq:1.V-B}, we recall that
$E^{\mu_{N,m}}[V(t,\cdot)]=0$ and use Lemma \ref{lem:1.11} to obtain
\begin{align}  \label{eq:1.32}
\|(-L_0)^{-1/2}V\|_{L^p(\mu_{N,m})} \le 2 \sup_{\fa\not\equiv{\rm const}}
 \frac{E^{\mu_{N,m}}[V\fa]}{\|(-L_0)^{1/2}\fa\|_{L^q(\mu_{N,m})}},
\end{align}
where $1/p+1/q=1$.

With this, we can localize the argument.  Recall $\ell\in \N, \ell<N/2,
\La_{\ell}=[-\ell,\ell-1]$ (as before, we omit $\lq\lq\cap\Z$'') and 
$\mathcal{X}_{\ell,m} = \{\eta\in \R^{2\ell}; \eta^{(\ell)}=m\}
\cong \R^{2\ell-1}$, $\mathcal{X}_{N,m} = \{\eta\in \R^{N}; \eta^{(N)}=m\}$, 
$\mathcal{X}_N=\R^N$.  The next proposition is used in the proof of Lemma
\ref{lem:1.22} taking $r$ to be the sum of shifts of $f-E^{\mu_{\ell,\eta^{(\ell)}}}[f]$.
Namely, it replaces $f$ with $E^{\mu_{\ell,\eta^{(\ell)}}}[f]$ (recall
\eqref{mu-ellm} for $\mu_{\ell,m}$) under the large 
space-time sum and integral.

\begin{prop} (cf. Proposition 5.2 of \cite{BFS})   \label{prop:1.21}
Let $r=r(\eta)$ be a local $L^p(\mu_{N,m})$ function on $\mathcal{X}_{N,m}$
depending only on $\eta_{\La_{\ell}} = (\eta(x))_{x\in \La_{\ell}}$.  Suppose that
$E^{\mu_{N,m}}[r|\eta^{(\ell)}] \; (=E^{\mu_{\ell,\eta^{(\ell)}}}[r])
=0$, $\mu_{N,m}$-a.s.\ $\eta$.  Then, for $p, q>1$ such that
$1/p+1/q=1$, $p'>p$ and $\fa=\fa(\eta) \in C^1(\R^N)$, we have 
$$
|E^{\mu_{N,m}}[r\fa] |\le C\ell \|r\|_{L^{p'}(\mu_{N,m})} 
\|\nabla_{\La_\ell}^* D_\eta \fa\|_{L^q(\mu_{N,m})},
$$
for some $C=C_{p,p'}>0$ uniform in $N$, $\ell$ and $m$,
where $\|\nabla_{\La_\ell}^* D_\eta \fa\|_{L^q(\mu_{N,m})}$ denotes
the $L^q(\mu_{N,m})$-norm of the Euclidean norm of 
$\nabla_{\La_\ell}^* D_\eta \fa \in \R^{2\ell-1}$ defined as in
\eqref{eq:4.15-A} and \eqref{eq:4.15-B}.
\end{prop}

\begin{proof}
The variable $\eta\equiv \eta_{\La_{\ell}}\in \R^{2\ell}$ is decomposed 
as a skew product
$(M,\eta) \in \R \times \mathcal{X}_{\ell,M}$  with $M\equiv M(\eta,\ell):=\eta^{(\ell)}$
so a function $r$ on $\R^{2\ell}$  is described as
$$
r(\eta) = r(M,\eta) \equiv r_M(\eta).
$$
Note that for each $M\in \R$, $r_M(\cdot)$ is a function on $\mathcal{X}_{\ell,M}$.

Then, since $E^{\mu_{\ell,M}} [r_M(\cdot)]=0$, there exists a function $u=u_M(\eta)$
on $\mathcal{X}_{\ell,M}$  such that $r_M = -L_{0,\ell,M}u_M$ holds; i.e.\
$u_M = (-L_{0,\ell,M})^{-1} r_M$ which exists by the condition (SG).
We denote $L_{0,\ell}$ by $L_{0,\ell,M}$ to indicate that it is
considered on the space $\mathcal{X}_{\ell,M}$.  Therefore,
\begin{align*}
\big|E^{\mu_{N,m}}[r\fa] \big|
& = \big| E^{\mu_{N,m}}[ E^{\mu_{N,m}}[r\fa|\eta^{(\ell)}] ] \big| 
= \big| E^{\mu_{N,m}}[ E^{\mu_{\ell,M}}[r_M\fa] ] \big|   \\
& = \big| E^{\mu_{N,m}}[ E^{\mu_{\ell,M}}[ (-L_{0,\ell,M}u_M) \fa] ] \big|   \\
& = \big| E^{\mu_{N,m}}[ E^{\mu_{\ell,M}}[ (-L_{0,\ell,M})^{1/2}  u_M \cdot 
(-L_{0,\ell,M})^{1/2}  \fa] ] \big|   \\
& \le E^{\mu_{N,m}} \big[ \|(-L_{0,\ell,M})^{1/2} u_M \|_{L^p(\mu_{\ell,M})}
\|(-L_{0,\ell,M})^{1/2}  \fa\|_{L^q(\mu_{\ell,M})} \big],
\end{align*}
recall that the average under $\mu_{N,m}$ is taken in $\eta$ through 
$M=M(\eta.\ell)$.  Here, by the weak $L^p$ Poincar\'e inequality 
(Lemma \ref{lem:1.15}, note $D_\phi f = \nabla_{\La_{\ell}}^* D_\eta f$ on
$\mathcal{X}_{\ell,M}$ for $f=f(\eta_{\La_{\ell}})$)  
on $L^p(\mu_{\ell,M})$ with $p'>p$ and then 
by the lower bound in the Littlewood-Paley-Stein inequality \eqref{eq:LPS-D}
in Proposition \ref{prop:LPS}, we have
\begin{align*}
\|(-L_{0,\ell,M})^{1/2} u_M \|_{L^p(\mu_{\ell,M})}
& = \|(-L_{0,\ell,M})^{-1/2} r_M \|_{L^p(\mu_{\ell,M})}  \\
& \le C_1 \ell \|\nabla_{\La_{\ell}}^* D_\eta(-L_{0,\ell,M})^{-1/2} r_M \|_{L^{p'}(\mu_{\ell,M})}  \\
& \le C_2 \ell \|(-L_{0,\ell,M})^{1/2}(-L_{0,\ell,M})^{-1/2} r_M\|_{L^{p'}(\mu_{\ell,M})}  \\
& = C_2 \ell \|r_M \|_{L^{p'}(\mu_{\ell,M})},
\end{align*}
for $2<p<p'<\infty$, where $C_1, C_2>0$ are uniform in $\ell$ and $M$.
On the other hand, by the upper bound 
in the Littlewood-Paley-Stein inequality \eqref{eq:LPS-D} 
for $-L_{0,\ell,M}$ on $\La_{\ell}$ 
with the measure $\mu_{\ell,M}$,
\begin{align*}
\|(-L_{0,\ell,M})^{1/2}  \fa\|_{L^q(\mu_{\ell,M})}
\le C_3 \|\nabla_{\La_{\ell}}^* D_\eta \fa\|_{L^q(\mu_{\ell,M})},
\end{align*}
where $C_3>0$ is uniform in $\ell$ and $M$.
Thus, we obtain by applying H\"older's inequality noting $p'>p$,
\begin{align*}
\big|E^{\mu_{N,m}}[r\fa] \big|
& \le C_2 C_3 \ell E^{\mu_{N,m}} \Big[ \|r_M \|_{L^{p'}(\mu_{\ell,M})}
\|\nabla_{\La_{\ell}}^* D_\eta \fa\|_{L^q(\mu_{\ell,M})} \Big] \\
& \le C_2 C_3 \ell E^{\mu_{N,m}} \Big[ \|r_M \|_{L^{p'}(\mu_{\ell,M})}^p\Big]^{1/p}
E^{\mu_{N,m}} \Big[ \|\nabla_{\La_{\ell}}^* D_\eta \fa\|_{L^q(\mu_{\ell,M})}^q \Big]^{1/q} \\
& \le C_2 C_3 \ell E^{\mu_{N,m}} \Big[ \|r_M \|_{L^{p'}(\mu_{\ell,M})}^{p'}\Big]^{1/p'}
E^{\mu_{N,m}} \Big[ E^{\mu_{\ell,M}} \big[ |\nabla_{\La_{\ell}}^* D_\eta \fa|^q \big]\Big]^{1/q} \\
& = C_2 C_3 \ell E^{\mu_{N,m}} \big[ |r |^{p'}\big]^{1/p'}
E^{\mu_{N,m}} \Big[ |\nabla_{\La_{\ell}}^* D_\eta \fa|^q \Big]^{1/q} \\
& = C_2 C_3 \ell \|r \|_{L^{p'}(\mu_{N,m})}
\|\nabla_{\La_{\ell}}^* D_\eta \fa\|_{L^q(\mu_{N,m})}.
\end{align*}
This completes the proof of the proposition.
\end{proof}

The steps to complete the proof of Theorem \ref{BG-2-B} will
proceed along with Lemmas \ref{lem:1.22}, 
\ref{globalrenormalization}, \ref{globaltwo-blocks} 
and \ref{EE_1block} similarly to \cite{GoJ} and \cite{BFS}.  
We need to extend them to the $L^p$ setting.
Recall that a function $h$ on $[0,T]\times\T_N$ and $u_0\in\R$ are
given as in Theorem \ref{BG-2-B}.

\begin{lem} \label{lem:1.22}
(One-block estimate, cf.\ Lemma 5.4 of \cite{BFS})
Let $f=f(\eta)$ be a local $L^{p'}(\mu_{N,m})$ function on $\mathcal{X}_{N,m}$
with $p'>p$ depending 
only on $\eta_{\La_{\ell_0}} = (\eta(x))_{x\in \La_{\ell_0}}$.  
Then, there exists $C>0$, uniform in $N, m, \ell$ and $\ga$,
such that, for $\ell\ge \ell_0$ and $p\ge 2$, we have
\begin{align}  \label{eq:L5.2-1}
E_{\mu_{N,m}} & \Bigg[ \sup_{0\le t \le T} \Big| \int_0^t \sum_{x\in \T_N}
h(s,x) \Big\{ f(\t_x \eta_s^N)- E_{\mu_{N,m}} [f(\t_x\eta_s^N)|\eta_s^{N,(\ell)}(x)]\Big\} ds
\Big|^p \Bigg] \\
& \le C T ^{(p-2)/2} N^{-p/2} \ell^{3p/2} \|f\|_{L^{p'}(\mu_{N,m})}^p 
\int_0^T \frac1N \sum_{x\in \T_N}|h(t,x)|^p dt.  \notag
\end{align}
From the uniformity of $C$ in $m$, we also have for a local $L^{p'}(\mu_{u_0})$
function $f$ on $\mathcal{X}_{N}$ depending only on $\eta_{\La_{\ell_0}}$,
\begin{align}  \label{eq:L5.2-2}
E_{\mu_{u_0}} & \Bigg[ \sup_{0\le t \le T} \Big| \int_0^t \sum_{x\in \T_N}
h(s,x) \Big\{ f(\t_x \eta_s^N)- E_{\mu_{u_0}} [f(\t_x\eta_s^N)|\eta_s^{N,(\ell)}(x)]\Big\} ds
\Big|^p \Bigg] \\
& \le C T ^{(p-2)/2} N^{-p/2} \ell^{3p/2} \|f\|_{L^{p'}(\mu_{u_0})}^p 
\int_0^T \frac1N \sum_{x\in \T_N}|h(t,x)|^p dt.  \notag
\end{align}
\end{lem}

Note that the condition $\tilde f(a_0)=0$ is unnecessary in Lemma 5.4 of \cite{BFS}.
Similarly, in our setting, the condition $E^{\mu_{N,m}}[f] =0$ is unnecessary.

\begin{proof}
Set
$$
r(t,\eta) :=\sum_{x\in \T_N} h(t, x) r_{x,\ell}(\eta)
\equiv \sum_{x\in \T_N}
h(t,x) \Big\{ f(\t_x \eta)- E^{\mu_{N,m}} [f(\t_x\eta)|\eta^{(\ell)}(x)]\Big\}.
$$ 
Then, we have
$$
E^{\mu_{N,m}}[r(t,\cdot)] =\sum_{x\in \T_N}
h(t,x) \Big\{E^{\mu_{N,m}}[ f]- E^{\mu_{N,m}}\big[ E^{\mu_{N,m}} [f|\eta^{(\ell)}] \big]\Big\}
=0.
$$
Therefore, by \eqref{eq:1.V-B} and \eqref{eq:1.32}, the left-hand side in 
the estimate \eqref{eq:L5.2-1} in this lemma is bounded by
\begin{align} \label{eq:1.33}
2C_p N^{-p} T^{(p-2)/2} \int_0^T \Big\{  \sup_{\fa\not\equiv{\rm const}}
 \frac{E^{\mu_{N,m}}[r(t,\cdot) \fa]}{\|(-L_0)^{1/2}\fa\|_{L^q(\mu_{N,m})}} \Big\}^p dt.
\end{align}
Here, for $\ell\ge \ell_0$, by Proposition \ref{prop:1.21} (shifting by $x$)
applied for each $r_{x,\ell}$ in $r(t,\cdot)$, 
noting $E^{\mu_{r,m}}[r_{x,\ell}|\eta^{(\ell)}(x)]=0$, we have
\begin{align}  \label{eq:1.34}
\big| E^{\mu_{N,m}}&[r(t,\cdot) \fa] \big|
 = \Bigg| \sum_{x\in \T_N} h(t,x) E^{\mu_{N,m}}[r_{x,\ell} \fa] \Bigg|\\
& \le \sum_{x\in \T_N} |h(t,x)| \cdot C\ell \|r_{0,\ell}\|_{L^{p'}(\mu_{N,m})} 
\|\nabla_{\La_{\ell}+x}^* D_\eta \fa\|_{L^q(\mu_{N,m})}   \notag   \\
& \le 2C\ell \| f\|_{L^{p'}(\mu_{N,m})} \sum_{x\in \T_N} |h(t,x)|  \; 
\Big\| \Big(\sum_{y\in \La_{\ell}} \big(\nabla^* D_\eta \fa(x+y)\big)^2\Big)^{1/2} 
\Big\|_{L^q(\mu_{N,m})},
\notag
\end{align}
recall that $C>0$ is uniform in $N$ and $m$.
The sum in $y: -\ell +1\le y \le \ell-1$ is bounded by that in $y\in \La_\ell$.

We estimate $a(x)\equiv a(x,\eta) :=\nabla^* D_\eta \fa(x)$, noting $q/2<1$
and first assuming $p>2$ (i.e.\ $p\not=2$), by H\"older's inequality,
\begin{align*}
\sum_{x\in \T_N} \Big( \sum_{y\in \La_{\ell}} a(x+y)^2 \Big)^{q/2}
& \le \Big( \sum_{x\in \T_N} 1^{\bar p}\Big) ^{1/\bar p}
 \Big( \sum_{x\in \T_N}\sum_{y\in \La_{\ell}} a(x+y)^2 \Big)^{q/2}\\
& = N^{(p-2)/2(p-1)} \Big( 2\ell \sum_{x\in \T_N} a(x)^2 \Big)^{q/2},
\end{align*}
where $\bar p$ is defined by $1/{\bar p}+{q}/2=1$ so  $\bar p=2/(2-q)$
and, recalling $q=p/(p-1)$, we further have $\bar p = 2(p-1)/(p-2)$.
Thus, we obtain
\begin{align}  \label{eq:1.35}
\sum_{x\in \T_N} &
\Big\| \Big(\sum_{y\in \La_{\ell}} \big(\nabla^* D_\eta \fa(x+y)\big)^2
\Big)^{1/2} \Big\|_{L^q(\mu_{N,m})}^q  \\
&  \le N^{(p-2)/2(p-1)} (2\ell)^{q/2}
\Big\| \Big(\sum_{x\in \T_N } \big(\nabla^* D_\eta \fa(x) \big)^2\Big)^{1/2} 
\Big\|_{L^q(\mu_{N,m})}^q.
\notag
\end{align}
Note that \eqref{eq:1.35} is obvious when $p=q=2$.

We now return to \eqref{eq:1.34} and use the equality
\begin{align}  \label{eq:ab}
ab= \inf_{\k>0} \Big( \frac{(\k a)^p}p + \frac{(b/\k)^q}q \Big), \quad a,b>0,
\end{align}
for $a= |h(t,x)|$ and $b= \Big\| \Big(\sum_{y\in \La_{\ell}} 
\big(\nabla^* D_\eta \fa(x+y)\big)^2\Big)^{1/2} \Big\|_{L^q(\mu_{N,m})}$.
For \eqref{eq:ab}, $\lq\lq \le$" is clear, while take
$\k= b^{1/p}/a^{1/q}$ for $\lq\lq\ge$".
Noting $\sum_{x\in \T_N}\inf_{\k>0} \le \inf_{\k>0} \sum_{x\in \T_N}$, we obtain
from \eqref{eq:1.34} and \eqref{eq:1.35}
\begin{align}  \label{eq:1.36}
\big|E^{\mu_{N,m}}[r(t,\cdot)\fa] \big|
\le & C_1 \ell \| f\|_{L^{p'}(\mu_{N,m})} \\
& \times \inf_{\k>0}
\Big\{ \k^p \sum_{x\in \T_N} |h(t,x)|^p + \k^{-q}
N^{(p-2)/2(p-1)} (2\ell)^{q/2} \cdot \Phi  \Big\},
\notag
\end{align}
where
$$
\Phi := \Big\| \Big(\sum_{x\in \T_N } \big(\nabla^* D_\eta \fa(x)\big)^2\Big)^{1/2} 
\Big\|_{L^q(\mu_{N,m})}^q.
$$
Take 
$$
\k= \left(\frac{N^{(p-2)/2(p-1)} (2\ell)^{q/2} \cdot \Phi}{ \sum_{x\in \T_N} |h(t,x)|^p}\right)^{1/(pq)}.
$$
Then, recalling $1/p + 1/q=1$, the two terms in the infimum in $\k>0$ balance,
and \eqref{eq:1.36} is further estimated as
$$
\big| E^{\mu_{N,m}}[r(t,\cdot)\fa] \big|
\le C_2\ell \| f\|_{L^{p'}(\mu_{N,m})} \Big(\sum_{x\in \T_N} |h(t,x)|^p\Big)^{1/p}
N^{(p-2)/(2p)} \ell^{1/2} \Phi^{1/q}.
$$

On the other hand,  by the lower bound in \eqref{eq:LPS-D} in
Proposition \ref{prop:LPS} (2) on $\T_N$,
the denominator in the supremum in \eqref{eq:1.33} is estimated from below as
\begin{align*}  
\|(-L_0)^{1/2}\fa\|_{L^q(\mu_{N,m})} 
& \ge c \| \nabla_{\T_N}^* D_\eta \fa \|_{L^q(\mu_{N,m})}   \\
& = c \Big\| \Big(\sum_{x\in \T_N } \big(\nabla^* D_\eta \fa(x)\big)^2
\Big)^{1/2} \Big\|_{L^q(\mu_{N,m})}
= c \Phi^{1/q},
\end{align*}
where $c=c_q>0$ is uniform in $N$ and $m$; note $1<q\le 2$ (cf.\ the last
comment in Remark \ref{rem:4.2-C}).
Thus, returning to \eqref{eq:1.33}, it is further bounded by 
\begin{align} \label{eq:5.10-B}
2C_p N^{-p} T^{(p-2)/2} \int_0^T \Big\{  c^{-1}C_2\ell 
\| f\|_{L^{p'}(\mu_{N,m})} \Big(\sum_{x\in \T_N} |h(t,x)|^p\Big)^{1/p}
N^{(p-2)/(2p)} \ell^{1/2} \Big\}^p dt
\end{align}
and, recalling that $C_p, c, C_2>0$ are all uniform in $N,m,\ell,\ga$,
we obtain the first estimate \eqref{eq:L5.2-1} under $\mu_{N,m}$.

To show the second estimate \eqref{eq:L5.2-2}
under $\mu_{u_0}$, we note the uniformity of $C>0$ in \eqref{eq:L5.2-1}
in $m\in \R$ and, regarding $m=\eta^{(N)}$ as random, we integrate 
\eqref{eq:L5.2-1} in $m$ under
$\mu_{u_0}(\eta^{(N)}\in dm)$.  Then, we first have
\begin{align}  \label{eq:superpose}
\int_\R \mu_{N,m}(\cdot) \mu_{u_0}(\eta^{(N)}\in dm) = \mu_{u_0}(\cdot).
\end{align}

Next, inside the expectation of \eqref{eq:L5.2-1}, replacing $m$ by 
$\eta^{(N)}$, we have the term
\begin{align}  \label{eq:5.7-A}
E^{\mu_N, \eta^{(N)}}[f(\t_x\eta)|\eta^{(\ell)}(x)]
= E^{\mu_{u_0}(\cdot| \eta^{(N)})}[f(\t_x\eta)|\eta^{(\ell)}(x)].
\end{align}
However, this coincides with $E^{\mu_{u_0}}[f(\t_x\eta)|\eta^{(\ell)}(x)]$.
To show this, since this function is $\si(\eta^{(\ell)}(x))$-measurable, we only
need to show
\begin{align}  \label{eq:1.cond-B}
E^{\mu_{u_0}(\cdot| \eta^{(N)})}[f(\t_x\eta) \Phi ]
= E^{\mu_{u_0}(\cdot| \eta^{(N)})}[E^{\mu_{u_0}}[f(\t_x\eta)|\eta^{(\ell)}(x)]\Phi ]
\end{align}
for any $\si(\eta^{(\ell)}(x))$-measurable bounded function $\Phi$;
recall that $f$ depends only on $\eta_{\La_{\ell_0}}$ and $\ell\ge \ell_0$.  Then,
the right-hand side of \eqref{eq:1.cond-B} is equal to
\begin{align} \label{eq:5.8-B}
&E^{\mu_{u_0}(\cdot| \eta^{(N)})}[E^{\mu_{u_0}}[f(\t_x\eta)\Phi|\eta^{(\ell)}(x)]] \\
&\quad = E^{\mu_{u_0}}[E^{\mu_{u_0}}[f(\t_x\eta)\Phi|\eta^{(\ell)}(x)]| \eta^{(N)}] 
 \notag  \\
&\quad = E^{\mu_{u_0}}[E^{\mu_{u_0}}[f(\t_x\eta)\Phi|
\eta^{(\ell)}(x), \eta_{\La_\ell(x)^c}]| \eta^{(N)}]     \notag  \\
& \quad = E^{\mu_{u_0}}[f(\t_x\eta)\Phi| \eta^{(N)}]  \notag
\end{align}
where $\La_\ell(x)= \La_\ell+x \subset \T_N$,
$\La_\ell(x)^c = \T_N\setminus \La_\ell(x)$, $\eta_{\La_\ell(x)^c}
= (\eta(y))_{y \in \La_\ell(x)^c}$.
The second equality in \eqref{eq:5.8-B} follows from the fact that
$f(\t_x\eta)\Phi$ is $\si(\eta_{\La_\ell(x)})$-measurable and $\mu_{u_0}$
is a product measure, while the last equality follows from
$\si(\eta^{(N)}) \subset \si(\eta^{(\ell)}(x), \eta_{\La_\ell(x)^c})$.
However, the right-hand side of \eqref{eq:5.8-B} 
is equal to the left-hand side of \eqref{eq:1.cond-B}.  Thus, 
 we can replace  $E^{\mu_{N,m}}[f(\t_x\eta)|\eta^{(\ell)}(x)]$
inside the expectation of \eqref{eq:L5.2-1} with
$E^{\mu_{u_0}}[f(\t_x\eta)|\eta^{(\ell)}(x)]$ as in  \eqref{eq:L5.2-2}.

Finally, we integrate the right-hand side of the estimate \eqref{eq:L5.2-1}
in $m$ by $\mu_{u_0}(\eta^{(N)}\in dm)$.  Then, noting that the function
$\psi(x)=x^{p/p'}, x \ge 0$, is concave since $p<p'$, by Jensen's inequality
and recalling \eqref{eq:superpose},
we have
\begin{align*}
\int_\R \|f\|_{L^{p'}(\mu_{N,m})}^p \mu_{u_0}(\eta^{(N)}\in dm)
& = \int_\R \mu_{u_0}(\eta^{(N)}\in dm)\Big( \int |f(\eta)|^{p'} d \mu_{N,m} \Big)^{p/p'}  \\
& \le   \Big( \int_\R  \mu_{u_0}(\eta^{(N)}\in dm)\int |f(\eta)|^{p'} d \mu_{N,m} \Big)^{p/p'}  \\
& = \Big( \int |f|^{p'} d\mu_{u_0}\Big)^{p/p'} = \|f\|_{L^{p'}(\mu_{u_0})}^p.
\end{align*}
This concludes the proof of the second estimate \eqref{eq:L5.2-2}
under $\mu_{u_0}$.
(Note that the local uniformity in $m$ is sufficient, since the tail
of $\mu_{u_0}$ decays exponentially fast by the large deviation estimate.)
\end{proof}

We will improve the constant $N^{-p/2}\ell^{3p/2}$ obtained
in  \eqref{eq:L5.2-2} in Lemma \ref{lem:1.22} to $N^{-p/2}\ell^{p/2}$ for
 the first estimate \eqref{eq:L5.4-1} and $N^{-p/2}\ell^p$ for the second 
 estimate \eqref{eq:L5.4-2} in Lemma \ref{globaltwo-blocks} based on 
an iteration argument prepared in Lemma \ref{globalrenormalization}.
During the proof, we use Lemma \ref{lem:1.22} only for $\ell=\ell_0$.  
Then, we have $N^p\ell^{-3p/2}$ 
for the first estimate \eqref{eq:L5.5-1} and $N^p\ell^{-p}$ for the second 
estimate \eqref{eq:L5.5-2} in Lemma \ref{EE_1block}.

\begin{lem}  \label{globalrenormalization} (Iteration lemma,
cf.\ Lemma 5.5 of \cite{BFS})
Let $p\ge 4$ and let $f$ be a local $L^{p'}(\mu_{u_0})$ function on
$\mathcal{X}_N$ for some $p'>p$ depending only on $\eta_{\La_{\ell_0}}$ 
such that $\tilde f(u_0)=0$ and $\tilde f'(u_0)=0$.
Then, there exists a constant $C=C(u_0, \ell_0)>0$ such that, for $\ell \geq \ell_0$, 
and a function $h$ on $[0,T]\times \T_N$, we have
\begin{align}  \label{eq:L5.3-1}
&E_{\mu_{u_0}}\Big[\sup_{0\leq t\leq T}\Big|\int_0^t\sum_{x\in{\T_N}} h(s,x) \\
& \hskip 20mm \times \Big\{ E_{\mu_{u_0}}[f(\tau_x\eta_s^N)|\eta_s^{N,(\ell)}(x)] 
-E_{\mu_{u_0}}[f(\tau_x \eta_s^N)|\eta_s^{N,(2\ell)}(x)]\Big\} ds \Big|^p\Big] 
\notag  \\
&\hskip 10mm \leq \ CT^{(p-2)/2} N^{-p/2} \ell^{p/2} \|f\|^p_{L^{p'}(\mu_{u_0})}
\int_0^T \frac{1}{N}\sum_{x\in{\T_N}} |h(t,x)|^p dt.
\notag
\end{align}
On the other hand, when only $\tilde f(u_0)=0$ is known, we have
\begin{align}  \label{eq:L5.3-2}
&E_{\mu_{u_0}}\Big[\sup_{0\leq t\leq T}\Big|\int_0^t\sum_{x\in{\T_N}} h(s,x)\\
&\hskip 20mm  \times \Big\{E_{\mu_{u_0}}[f(\tau_x\eta_s^N)|\eta_s^{N,(\ell)}(x)]
-E_{\mu_{u_0}}[f(\tau_x\eta_s^N)|\eta_s^{N,(2\ell)}(x)]\Big\} ds \Big|^p\Big] 
\notag \\
&\hskip 10mm  \leq \ CT^{(p-2)/2} N^{-p/2} \ell^p \|f\|^p_{L^{p'}(\mu_{u_0})}
\int_0^T \frac{1}{N}\sum_{x\in{\T_N}} |h(t,x)|^p dt.
\notag
\end{align}
\end{lem}

\begin{proof}  
We follow similar steps as in the proof of Lemma \ref{lem:1.22} based on
\eqref{eq:1.V-B} and \eqref{eq:1.32}, replacing $r(t,\eta)$ with 
\begin{align*}
\hat r(t,\eta) := & \sum_{x\in \T_N} h(t,x)  \hat r_{x,\ell}(\eta)  \\
\equiv & \sum_{x\in \T_N} h(t,x)  \Big\{E^{\mu_{u_0}}[f(\tau_{x}\eta)|\eta^{(\ell)}(x)]
-E^{\mu_{u_0}}[f(\tau_{x}\eta)|\eta^{(2\ell)}(x)]\Big\}.
\end{align*}
Note that  in the above definition, $\mu_{u_0}$ can be replaced by $\mu_{N,m}$,
as we noted below \eqref{eq:5.8-B}, which follows from \eqref{eq:5.7-A}
and \eqref{eq:1.cond-B}.  Therefore, we have $E^{\mu_{N,m}}[\hat r(t,\eta)]=0$.

Let us start by showing the first estimate \eqref{eq:L5.3-1}.
In \eqref{eq:1.34}, we estimated as $\|r_{0,\ell}\|_{L^{p'}(\mu_{N,m})} \le 2 
\|f\|_{L^{p'}(\mu_{N,m})}$ in the third line.  Here, instead, to estimate 
$|E^{\mu_{N,m}}[\hat r(t,\cdot)\fa]|$, taking $p''$ such that $p<p''<p'$, 
we replace $\|r_{0,\ell}\|_{L^{p'}(\mu_{N,m})}$ by $\|\hat r_{0,\ell}\|_{L^{p''}(\mu_{N,m})}$
in the second line of \eqref{eq:1.34}.  Then, keeping it in the following 
calculation up to \eqref{eq:5.10-B}, we can estimate \eqref{eq:1.33}
with $r(t,\cdot)$ replaced by $\hat r(t,\cdot)$ by \eqref{eq:5.10-B} with 
$2\|f\|_{L^{p'}(\mu_{N,m})}$ replaced by $\|\hat r_{0,\ell}\|_{L^{p''}(\mu_{N,m})}$.
Therefore, instead of \eqref{eq:L5.3-1}, we obtain
\begin{align}  \label{eq:L5.3-1-B}
&E_{\mu_{N,m}}\Big[\sup_{0\leq t\leq T}\Big|\int_0^t\sum_{x\in{\T_N}} h(s,x) \\
& \hskip 20mm \times \Big\{ E_{\mu_{u_0}}[f(\tau_x\eta_s^N)|\eta_s^{N,(\ell)}(x)] 
-E_{\mu_{u_0}}[f(\tau_x \eta_s^N)|\eta_s^{N,(2\ell)}(x)]\Big\} ds \Big|^p\Big] 
\notag  \\
&\hskip 10mm \leq \ CT^{(p-2)/2} N^{-p/2} \ell^{3p/2} 
\|\hat r_{0,\ell}\|^p_{L^{p''}(\mu_{N,m})}
\int_0^T \frac{1}{N}\sum_{x\in{\T_N}} |h(t,x)|^p dt,
\notag
\end{align}
where $C>0$ is uniform in $N, m, \ell$ and $\ga$.
Then, we integrate both sides of \eqref{eq:L5.3-1-B} with respect to $m$,
recalling \eqref{eq:superpose}.  Thus, by H\"older's inequality noting $p<p''$,
we see that the left-hand side of \eqref{eq:L5.3-1} is bounded by
\begin{align}  \label{eq:18-Q}
CT^{(p-2)/2} N^{-p/2} \ell^{3p/2} 
\|\hat r_{0,\ell}\|^p_{L^{p''}(\mu_{u_0})}
\int_0^T \frac{1}{N}\sum_{x\in{\T_N}} |h(t,x)|^p dt.
\end{align}
We will further estimate $\|\hat r_{0,\ell}\|_{L^{p''}(\mu_{u_0})}$ as
\begin{align}  \label{eq:1.37-B}
\|\hat r_{0,\ell}\|_{L^{p''}(\mu_{u_0})} \equiv
\Big\|E^{\mu_{u_0}}[f(\eta)|\eta^{(\ell)}]-E^{\mu_{u_0}}[f(\eta)|\eta^{(2\ell)}]\Big\|_{L^{p''}(\mu_{u_0})}
\le C \|f\|_{L^{p'}(\mu_{u_0})}\ell^{-1}.
\end{align}
Once this is shown, it leads to the estimate \eqref{eq:L5.3-1}.  Note that
a new factor $(\ell^{-1})^p$ appears from \eqref{eq:1.37-B} for 
$\|\hat r_{0,\ell}\|^p_{L^{p''}(\mu_{u_0})}$
in \eqref{eq:18-Q}.  Compare this to the estimate in Lemma \ref{lem:1.22}.

To prove \eqref{eq:1.37-B}, we bound its left-hand side by
\begin{align*}
& \Big\|E^{\mu_{u_0}}\Big[f(\eta) - \tfrac{1}{2} \tilde f''(u_0)\Big\{(\eta^{(\ell)} - u_0)^2
  -\tfrac{{\rm var}(u_0) }{2\ell +1}\Big\}\Big|\eta^{(\ell)}\Big]\Big\|_{L^{p''}(\mu_{u_0})}\\
&   +\Big\|E^{\mu_{u_0}}\Big[f(\eta)
 - \tfrac{1}{2} \tilde f''(u_0) \Big\{(\eta^{(2\ell)} - u_0)^2  -\tfrac{{\rm var}(u_0)}{4\ell+1}\Big\}\Big|\eta^{(2\ell)}\Big]\Big\|_{L^{p''}(\mu_{u_0})}\\
&  + \Big\|\tfrac{1}{2} \tilde f''(u_0)\Big\{E^{\mu_{u_0}}\Big[(\eta^{(\ell)} - u_0)^2 
-\tfrac{{\rm var}(u_0)}{2\ell +1}\Big|\eta^{(\ell)}\Big] \\
& \hskip 20mm
+ E^{\mu_{u_0}}\Big[(\eta^{(2\ell)} - u_0)^2 -\tfrac{{\rm var}(u_0)}{4\ell+1}\Big|\eta^{(2\ell)}\Big]\Big\}
\Big\|_{L^{p''}(\mu_{u_0})}.
\end{align*}
The first two terms are  of order $\|f\|_{L^{p'}(\mu_{u_0})} \times O(\ell^{-3/2})$
by Proposition \ref{thm:EE}; note $p'>p''$.  For the last term, by a $p''$-moment
bound of $(\eta^{(\ell')}-u_0)^2$ with variously $\ell'=\ell$ and $2\ell$ and using that 
$|\tilde f''(u_0)| \le C \|f\|_{L^2(\mu_{u_0})}$,  it is of order 
$\|f\|_{L^2(\mu_{u_0})}\times O(\ell^{-1})$.  This proves \eqref{eq:1.37-B}.

To show the second estimate \eqref{eq:L5.3-2}, under the condition 
$\tilde f(u_0)=0$ only, we claim the bound
\begin{align*}
\|\hat r_{0,\ell}\|_{L^{p''}(\mu_{u_0})}\le C \|f\|_{L^{p'}(\mu_{u_0})}\ell^{-1/2}.
\end{align*}
Indeed, we bound similarly as above.  The first two terms (modified properly)
are bounded by Proposition \ref{thm:EE} and they are
of order $\|f\|_{L^{p'}(\mu_{u_0})}\times O(\ell^{-1})$, while the last term 
(also modified properly) is of order $\|f\|_{L^2(\mu_{u_0})}\times O(\ell^{-1/2})$.
Therefore, we now have a new factor $(\ell^{-1/2})^p$ in the final form 
compared to that of Lemma \ref{lem:1.22} for the second estimate. 
\end{proof}

The next lemma replaces $\ell_0$ with $\ell \ge \ell_0$ and Lemma \ref{EE_1block} will apply for large $\ell$. 

\begin{lem}  \label{globaltwo-blocks} (Two-blocks estimate,
cf.\ Lemma 5.6 of \cite{BFS})
Let $p\ge 4$ and let $f$ be a local $L^{p'}(\mu_{u_0})$ function on
$\mathcal{X}_N$ for some $p'>p$ depending only on $\eta_{\La_{\ell_0}}$ 
such that $\tilde f(u_0)=0$ and $\tilde f'(u_0)=0$.  Then, there exists 
a constant $C = C(u_0,\ell_0)$ such that, for $\ell \geq  \ell_0$,  and 
a function $h$ on $[0,T]\times \T_N$, we have
\begin{align}  \label{eq:L5.4-1}
&E_{\mu_{u_0}}\Big[\sup_{0\leq t\leq T}\Big|\int_0^t\sum_{x\in{\T_N}} h(s,x)\\
& \hskip 20mm \times\Big\{
E_{\mu_{u_0}}[f(\tau_x\eta_s^N)|\eta_s^{N,(\ell_0)}(x)]
-E_{\mu_{u_0}}[f(\tau_x\eta_s^N)|\eta_s^{N,(\ell)}(x)]\Big\} ds \Big|^p\Big] 
\notag\\
&\hskip 10mm  \leq CT^{(p-2)/2} N^{-p/2}\ell^{p/2} \|f\|^p_{L^{p'}(\mu_{u_0})}
\int_0^T \frac1{N}\sum_{x\in{\T_N}} |h(t,x)|^p dt.
\notag
\end{align}
On the other hand, when only $\tilde f(u_0)=0$ is known, we have
\begin{align}  \label{eq:L5.4-2}
&E_{\mu_{u_0}}\Big[\sup_{0\leq t\leq T}\Big|\int_0^t\sum_{x\in{\T_N}} h(s,x) \\
& \hskip 20mm \times \Big\{E_{\mu_{u_0}}[f(\tau_x\eta_s^N)|\eta_s^{N,(\ell_0)}(x)]-E_{\mu_{u_0}}[f(\tau_x\eta_s^N)|\eta_s^{N,(\ell)}(x)]\Big\} ds \Big|^p\Big] 
\notag\\
&\hskip 10mm \leq CT^{(p-2)/2} N^{-p/2} \ell^p \|f\|^p_{L^{p'}(\mu_{u_0})}
\int_0^T \frac1N \sum_{x\in{\T_N}} |h(t,x)|^p dt.
\notag
\end{align}
\end{lem}

\begin{proof}  
Write $\ell = 2^{n}\ell_0 + r$ where $n\in \Z_+$ and $0\leq r\leq 2^{n}\ell_0 -1$.  Then,
\begin{align*}
& E_{\mu_{u_0}}[f(\tau_x\eta_s^N)|\eta_s^{N,(\ell_0)}(x)]
-E_{\mu_{u_0}}[f(\tau_x\eta_s^N)|\eta_s^{N,(\ell)}(x)]  \\
& \quad  = \sum_{i=0}^{n-1} \Big\{E_{\mu_{u_0}}[f(\tau_x\eta_s^N)|\eta_s^{N,(2^i \ell_0)}(x)]-E_{\mu_{u_0}}[f(\tau_x\eta_s^N)|\eta_s^{N,(2^{i+1}\ell_0)}(x)]\Big\}  \\
& \qquad + 
\ E_{\mu_{u_0}}[f(\tau_x\eta_s^N)|\eta_s^{N,(2^{n}\ell_0)}(x)]  
- E_{\mu_{u_0}}[f(\tau_x\eta_s^N)|\eta_s^{N,(\ell)}(x)].
\end{align*}
Now, by Minkowski's inequality and \eqref{eq:L5.3-1} in Lemma
\ref{globalrenormalization} (for the last term above, noting that $\ell \in 
[2^n\ell_0, 2^{n+1}\ell_0)$, we have a similar estimate to 
\eqref{eq:L5.3-1} replacing $\ell$ in \eqref{eq:L5.3-1} with $2^n\ell_0$ 
and replacing $2\ell$  in \eqref{eq:L5.3-1} with $\ell$
in the present setting), we obtain
that the left-hand side of \eqref{eq:L5.4-1} is bounded by 
\begin{align*}
& CT^{(p-2)/2} N^{-p/2} \left\{ 
\sum_{i=0}^{n-1} \big(2^{i}\ell_0\big)^{1/2} +\big(2^{n}\ell_0\big)^{1/2} 
 \right\}^p \|f\|^p_{L^{p'}(\mu_{u_0})} 
\int_0^T \frac1N \sum_{x\in \T_N} |h(t,x)|^p dt\\
& \le C_1 T^{(p-2)/2} N^{-p/2} \ell^{p/2} \|f\|^p_{L^{p'}(\mu_{u_0})} 
\int_0^T \frac1N\sum_{x\in \T_N} |h(t,x)|^p dt.
 \end{align*}
This concludes the proof of \eqref{eq:L5.4-1}.  
The second estimate \eqref{eq:L5.4-2}
is shown similarly from \eqref{eq:L5.3-2}.
\end{proof}

The last step is based on the equivalence of ensembles.
Recall that $E^{\mu_{u_0}}[f|\eta^{(\ell)}] = E^{\mu_{N,m}}[f|\eta^{(\ell)}]$
for $f$ depending only on $\eta_{\La_{\ell_0}}$ 
and $\ell\ge \ell_0$; see below \eqref{eq:5.8-B}.

\begin{lem}  (cf.\ Lemma 5.7 of \cite{BFS}) \label{EE_1block}
Let $p\ge 4$ and let $f$ be a local $L^{p'}(\mu_{u_0})$ function on
$\mathcal{X}_N$ for some $p'>p$ depending only on $\eta_{\La_{\ell_0}}$  
such that $\tilde f(u_0)=0$ and $\tilde f'(u_0)=0$.  Then, there exists a constant 
$C = C(u_0,\ell_0)$ such that, for $\ell \geq  \ell_0$, and a function
$h$ on $[0,T]\times \T_N$, we have
\begin{align}   \label{eq:L5.5-1}
&E_{\mu_{u_0}}\Big[\sup_{0\leq t\leq T}\Big|\int_0^t\sum_{x\in{\T_N}}h(s,x)\\
& \hskip 10mm \times
\Big\{E_{\mu_{u_0}}\big[f(\tau_x\eta_s^N)|\eta_s^{N,(\ell)}(x)\big] 
-\frac{1}{2} \tilde f''(u_0) \Big((\eta_s^{N,(\ell)}(x) - u_0)^2- \tfrac{{\rm var}(u_0) }{2\ell +1}\Big)\Big\} ds \Big|^p\Big] 
\notag \\
& \hskip 5mm\leq \ CT^{p-1} N^p \ell^{-3p/2} \|f\|^p_{L^{p'}(\mu_{u_0})}
\int_0^T \frac{1}{N}\sum_{x\in \T_N}|h(t,x)|^p dt.
\notag
\end{align}
On the other hand, when only $\tilde f(u_0)=0$ is known, we have
\begin{align}  \label{eq:L5.5-2}
&E_{\mu_{u_0}}\Big[\sup_{0\leq t\leq T}\Big|\int_0^t\sum_{x\in{\T_N}} h(s,x) \\
& \hskip 20mm \times\Big\{E_{\mu_{u_0}}[f(\tau_x\eta_s^N)|\eta_s^{N,(\ell)}(x)] 
-\tilde f'(u_0)
\big(\eta_s^{N,(\ell)}(x)-u_0\big) \Big\} ds\Big|^p\Big] 
\notag  \\
& \hskip 10mm
 \leq \ CT^{p-1} N^p \ell^{-p} \|f\|^{p}_{L^{p'}(\mu_{u_0})} \int_0^T \frac{1}{N}\sum_{x\in \T_N}|h(t,x)|^p dt.  \notag
\end{align}
\end{lem}

\begin{proof}
For the first estimate \eqref{eq:L5.5-1}, denoting the expression in
curly braces by $\psi_\ell(x,\eta_s^N)$, the left-hand side is bounded by
\begin{align*}
E_{\mu_{u_0}}&\Big[\Big(\int_0^T\sum_{x\in{\T_N}} |h(s,x)| \, |\psi_\ell(x,\eta_s^N)| ds \Big)^p\Big] \\
& \le T^{p/q} \int_0^T E_{\mu_{u_0}}\Big[\Big(\sum_{x\in{\T_N}} |h(s,x)| \, |\psi_\ell(x,\eta_s^N)|\Big)^p
\Big] ds\\
& \le T^{p/q} \int_0^T \Big(\sum_{x\in{\T_N}} |h(s,x)|^q\Big) ^{p/q}
\sum_{x\in{\T_N}} E_{\mu_{u_0}}\big[|\psi_\ell(x,\eta_s^N)|^p\big] ds
\end{align*}
by using H\"older's inequality twice, where $1/p+1/q=1$.  However, 
by stationarity and translation-invariance, 
the last expectation is equal to $E^{\mu_{u_0}}\big[|\psi_\ell(0,\eta)|^p\big]$, which is bounded by
$$
\Big(\frac{C\|f\|_{L^{p'}(\mu_{u_0})}}{\ell^{3/2}}\Big)^p
$$
by Proposition \ref{thm:EE}.  Thus, recalling $p/q=p-1$, the left-hand side
of \eqref{eq:L5.5-1} is bounded by
\begin{align*}
 T^{p-1}  &N \Big(\frac{C\|f\|_{L^{p'}(\mu_{u_0})}}{\ell^{3/2}}\Big)^p 
 \int_0^T \Big(\sum_{x\in{\T_N}} |h(t,x)|^q\Big)^{p/q} dt \\
& \le  C^p T^{p-1} N^p \ell^{-3p/2} \|f\|_{L^{p'}(\mu_{u_0})}^p
\int_0^T \frac1N \sum_{x\in{\T_N}} |h(t,x)|^p dt.
\end{align*}
The last line is shown by  H\"older's inequality noting $q\le p$.
This shows \eqref{eq:L5.5-1}.

The second estimate \eqref{eq:L5.5-2} is shown similarly.
\end{proof}

\begin{proof}[Proof of Theorem \ref{BG-2-B}]
Theorem \ref{BG-2-B} follows by first applying Lemma \ref{lem:1.22}
with $\ell=\ell_0$ and then Lemmas \ref{globaltwo-blocks} and \ref{EE_1block}.
\end{proof}

\appendix

\section{Discussion of the Littlewood-Paley-Stein inequality}
\label{sec:B}

The Littlewood-Paley-Stein inequality \eqref{eq:LPS-D} in Proposition 
\ref{prop:LPS} is stated in Shigekawa \cite{Shi}
Theorem 4.4 (4.16), $k=1$, for the Ornstein-Uhlenbeck generator $L_0$ on
an abstract Wiener space; i.e.\ for $L_0$ with linear drift.
This was extended to the distorted Laplacian on a Riemannian manifold by
Yoshida \cite{Y}.  See also \cite{Shi02}, \cite{SY}, \cite{Li}.

However, the estimates obtained in these papers are weaker than
\eqref{eq:LPS-D} as in the next Theorem \ref{thm:LPS}.  We state the 
estimate for the $\phi$-variables.  To rewrite it for the $\eta$-variables, 
we may replace $D_\phi$ by $\nabla^* D_\eta$ recalling \eqref{eq:d-phi}
or \eqref{eq:4.12-A}.

\begin{thm}  \label{thm:LPS}
Assume $V''\ge -\a$ for some $\a>0$ and $\int_\R e^{-V(\eta)} d\eta<\infty$.
Let $p\in (0,\infty)$.  Then, there exist constants $c=c_{p,\a}, C=C_{p,\a}>0$, 
which are uniform in $\ell$ and $m$, such that 
\begin{align} \label{eq:LPS-B}
c\big\{ \|f\|_{L^p(\bar\mu_{\ell,m})} + \|D_\phi f\|_{L^p(\bar\mu_{\ell,m})} \big\}
&\le \|(1-L_{0,\ell})^{1/2}f\|_{L^p(\bar\mu_{\ell,m})} \\
&\le 
C\big\{ \|f\|_{L^p(\bar\mu_{\ell,m})} + \|D_\phi f\|_{L^p(\bar\mu_{\ell,m})} \big\},
\notag
\end{align}
for $f \in W^{1,p}(\mathcal{Y}_{\ell,m}, \bar\mu_{\ell,m})$, where 
$D_\phi f= (\partial_{\phi(x)}f)_{x=-\ell+1}^{\ell-1]} \in \R^{2\ell-1}$.
\end{thm}

Note that, in \eqref{eq:LPS-D}, $(1-L_{0,\ell})^{1/2}$ is replaced by 
$(-L_{0,\ell})^{1/2}$
and the norm $\|f\|_{L^p(\bar\mu_{\ell,m})}$ is dropped from both sides
of \eqref{eq:LPS-B}.  The estimate \eqref{eq:LPS-B} is not sufficient for 
our purposes.

\begin{proof}
We apply Theorem 1.1 of \cite{Y}, taking $M=\mathcal{Y}_{\ell,m}\cong \R^{2\ell-1}$, 
$g=(\de^{ij})$ the usual flat metric, $F=M\times \R$, so that 
$\Ga(F) = C^\infty(M,\R)$, 
Ricci tensor Ric $\equiv 0$, $R_F\equiv 0$, $\rho(\phi)= H_{\ell,m}(\phi)$ given in
\eqref{eq:Hem}, $Au= \De u - \nabla\rho\cdot\nabla u$, $V\equiv 0$ in (1.2)
of \cite{Y} so that $\La = A \, (=2L_{0,\ell})$.
We check the condition $(\mathcal{M}')$ with $m(\cdot)$ defined in (1.4) of \cite{Y}.
Note that 
$$
\big((\nabla)^2\rho\big)(\phi) = \text{Hess}\, H_{\ell,m}(\phi)
\equiv (\partial_{\phi(x)}\partial_{\phi(y)} H_{\ell,m}(\phi))_{x,y\in [-\ell+1,\ell-1]}
\in \R^{2\ell-1} \otimes \R^{2\ell-1}.
$$
Then, according to the calculation in the proof of 
Proposition \ref{prop:LPS}, since $V''\ge -\a$, 
we have for $\xi=(\xi_x)_{x=-\ell+1}^{\ell-1}\in\R^{2\ell-1}
(=T_\phi\mathcal{Y}_{\ell,m}$, the tangent space at $\phi\in\mathcal{Y}_{\ell,m})$
with $\xi_{-\ell}=\xi_\ell=0$,
\begin{align*}
(\text{Hess}\, H_{\ell,m}\; \xi,\xi) 
\ge -\a \sum_{x=-\ell}^{\ell-1} (\xi_{x+1}-\xi_x)^2
 \ge -4\a \sum_{x=-\ell+1}^{\ell-1} \xi_x^2,
\end{align*}
since $(\xi_{x+1}-\xi_x)^2 \le 2(\xi_{x+1}^2+\xi_x^2)$.
By Theorem 1.1 of \cite{Y}, this shows \eqref{eq:LPS-B} with
constants $c, C>0$ depending only on $p$ and $\a$.
\end{proof}

\begin{rem}  \label{rem:1.5}
If we assume a stronger condition $\|V''\|_\infty+ \|V'''\|_\infty <\infty$, then
Theorem 1.2 of \cite{Y} is applicable and \eqref{eq:LPS-B} is shown by
taking $k=1$.  Note that the constants $c, C$
depend only on $p$, $k$, $\|(\nabla)^2\rho\|_{\infty,1}$, and we have
\begin{align*}
\|(\nabla)^2\rho\|_{\infty,1} = \|(\nabla)^2\rho\|_{\infty}+
\|(\nabla)^3\rho\|_{\infty}  \le 2(\|V''\|_\infty+ \|V'''\|_\infty),
\end{align*}
which is uniform in  $\ell$ and $m$.
\end{rem}

\section*{Acknowledgement}

The author thanks Nobuo Yoshida for his comments on the 
Littlewood-Paley-Stein inequality.

\end{document}